\documentclass[11pt]{article}
\usepackage{}
\usepackage{amsmath}
\usepackage{amsfonts}
\usepackage{graphicx}
\usepackage{amssymb}
\usepackage{color}

\newtheorem{condition**}{A*}
\newtheorem{condition***}{C*}
\newtheorem{condition*}{C}

\newtheorem{proposition}{Proposition}[section]

\newtheorem{definition}{Definition}[section]
\newtheorem{theorem}{Theorem}[section]
\newtheorem{lemma}{Lemma}[section]
\newtheorem{remark}{Remark}[section]

\def\FF{\mathbb F}\def\RR{\mathbb R}
\def\cF{\mathcal F}\def\al{\alpha}

\newenvironment{keywords}{{\bf Key words: }}{}
\newenvironment{AMS}{{\bf AMS subject classification: }}{}

\begin{document}

\title{Social Optima of Linear Forward-Backward Stochastic System\thanks{This work was firstly completed in May 2021, and we arrive the current version after many revisions. The first author was supported in part by the National Key R\&D Program of China (No. 2022YFA1006103), the NSFC for Distinguished Young Scholars (No. 61925306), the NSFC (No. 11831010), and the NSF of Shandong Province (Nos. ZR2020ZD24 and ZR2019ZD42).
 The second author was supported in part by the National Key R\&D Program of China (Nos. 2022YFA1006104, 2023YFA1009203), the Taishan Scholars Young Program of Shandong (No. TSQN202211032) and the Young Scholars Program of Shandong University.
 The third author was supported in part by the National Key R\&D Program of China (No. 2022YFA1006102), and the NSFC (No. 11831010).}}

\author{Guangchen Wang\thanks{School of Control Science and Engineering, Shandong University, Jinan, Shandong 250061, China
  (wguangchen@sdu.edu.cn).}
\qquad Shujun Wang\thanks{Corresponding author. School of Management, Shandong University, Jinan, Shandong 250100, China
  (wangshujun@sdu.edu.cn).}
\qquad Jie Xiong\thanks{Department of Mathematics and SUSTech International Center for Mathematics, Southern University of Science and Technology, Shenzhen, Guangdong 518055, China   (xiongj@sustech.edu.cn).}}

\date{}

\maketitle

\begin{abstract}
A linear quadratic (LQ) stochastic optimization system involving large population, which is driven by \emph{forward-backward stochastic differential equation (FBSDE)}, is investigated in this paper. Agents cooperate with each other to minimize the so-called social objective, which is rather different from mean field (MF) game. Employing forward-backward person-by-person optimality principle, we derive an auxiliary LQ control problem by decentralized information. A decentralized strategy is obtained by virtue of an MF-type forward-backward stochastic differential equation consistency condition. Applying \emph{Riccati equation} decoupling method, we solve the consistency condition system. We also verify the asymptotic social optimality in this framework.
\end{abstract}

\begin{keywords}  {Forward-backward stochastic differential equation, Riccati equation, social optima, mean field game.}\end{keywords}

\begin{AMS}
91A07, 91A15, 93E03, 93E20
\end{AMS}

\section{Introduction}
\setcounter{equation}{0}
\renewcommand{\theequation}{\thesection.\arabic{equation}}

In this section,  we first present some notation and introduce the main motivations of this work. Then a \emph{forward-backward stochastic LQ MF social optima problem} is posed, which will be investigated in this paper.

\subsection{Notation}

For $T>0$, let $(\Omega,\mathcal F,\{\mathcal F_t\}_{0\leq t\leq T},\mathbb P)$ be a complete probability space, on which a $N$-dimensional standard Brownian motion $\{W_i(t),1\leq i\leq N\}_{0\leq t\leq T}$ is defined. $W(t):=(W_1(t),\ldots,W_N(t))^\top$. Let $\FF=\{\mathcal F_t,\;t\in[0,T]\}$ denote the filtration generated by $\{W_i(s),\ \xi_i,\ 1\leq i\leq N\}_{0\leq s\leq t}$ and augmented by $\mathcal N_{\mathbb P}$ (which is the class of all $\mathbb P$-null sets of $\mathcal F$). Let $\mathcal F_t^i$ denote the augmentation of $\sigma\{W_i(s),\ \xi_i,\ 0\leq s\leq t\}$ by $\mathcal N_{\mathbb P}$, $1\leq i\leq N$. Here, $\xi_i,\ 1\leq i\leq N$ are initial values of states which will be defined later.

In this paper, the Euclidean inner product is denoted by $\langle\cdot,\cdot\rangle$. $A^\top$ stands for the transpose of a matrix (or vector) $A$. Let $\mathbb{S}^n$ denote the set of symmetric $n\times n$ matrices. If $A\in \mathbb{S}^n$ is positive (semi) definite, we write $A> (\geq)\ 0$. We write $A\gg 0$, if $M - \epsilon I \geq 0$ for some $\epsilon>0$. We introduce the following spaces:
\begin{itemize}
  \item $L_{\mathcal F}^2(\Omega;\mathbb R^n):=\Big\{\zeta:\Omega\rightarrow\mathbb R^n|\zeta\text{ is }\mathcal F\text{-measurable and }\mathbb E|\zeta|^2<\infty\Big\}$;
  \item $L_{\mathcal F}^\infty(\Omega;\mathbb R^{n}):=\Big\{\zeta:\Omega\rightarrow\mathbb R^{n}|\zeta\text{ is }\mathcal F\text{-measurable and  uniformly bounded}\Big\}$;
    \item $L^2_{\FF}(\Omega;C([0,T];\mathbb R^n)):=\Big\{\zeta(\cdot):\Omega\times[0,T]\rightarrow\mathbb R^n|\zeta(\cdot)$ \text{is continuous and}\ $\mathcal F_t$\text{-adapted}\\ \text{satisfying} $\mathbb E \Big[\sup\limits_{s\in[0,T]}|\zeta(s)|^2\Big]<\infty\Big\}$;
  \item $L_{\FF}^2(0,T;\mathbb R^n):=\Big\{\zeta(\cdot):\Omega\times[0,T]\rightarrow\mathbb R^n|\zeta(\cdot)\text{ is an }\mathcal F_t\text{-progressively measurable}$\\ $\text{process satisfying\ }      \mathbb E\int_0^T|\zeta(s)|^2ds<\infty\Big\}$;
  \item $L^2(0,T;\mathbb R^n):=\Big\{\zeta(\cdot):[0,T]\rightarrow\mathbb R^n|\int_0^T|\zeta(s)|^2ds<\infty\Big\}$;
  \item $L^\infty(0,T;\mathbb R^{n\times m}):=\Big\{\zeta(\cdot):[0,T]\rightarrow\mathbb R^{n\times m}|\zeta(\cdot)\text{ is uniformly bounded}\Big\}$.
   \end{itemize}

\subsection{Motivation}\label{sec1.2}
Consider a controlled large population (also called multi-agent) system in which the dynamic of the agent $\mathcal A_i$ is modelled by a stochastic differential equation (SDE)
\begin{equation}\left\{\label{eq00}\begin{aligned}
dX_i(t)=\ &b\Big(t,X_i(t),u_i(t),X^{(N)}(t)\Big)dt+\sigma\Big(t,X_i(t),u_i(t),X^{(N)}(t)\Big)dW_i(t),\\
X_i(0)=\ &x_{i0},
\end{aligned}\right.\end{equation}
with cost functional
\begin{equation}\label{eq01}\begin{aligned}
\mathcal J_i(x_{i0},u(\cdot))=\ &\mathbb E\left\{\int_0^TL(t,X_i(t),u_i(t),X^{(N)}(t))dt+\Phi(x_{i0},X_i(T))\right\},
\end{aligned}\end{equation}
where $X^{(N)}(\cdot)=\frac{1}{N}\sum_{i=1}^NX_i(\cdot)$ denotes the state-average of agents; $u(\cdot)=(u_1(\cdot)$, $\cdots,u_N(\cdot))$, $u_i\in \mathcal U_i:=\Big\{u_i(\cdot)|u_i(\cdot)\in L^2_{\FF}(0,T;\mathbb R^d)\Big\}$, $x_{i0}\in L_{\mathcal F_0}^2(\Omega;\mathbb R^n)$, $i=1,\cdots,N$.
Define
$$
\mathcal J_{soc}^{(N)}(u(\cdot))=\sum_{i=1}^N\mathcal J_i(x_{i0},u(\cdot))
$$
as the aggregated functional of $N$ agents. Then we can pose a classical MF optimal control problem.

\textbf{Problem 0.} Find a strategy $\bar u=(\bar u_1,\cdots,\bar u_N)$ where $\bar u_i(\cdot)\in \mathcal U_i$, $1\leq i\leq N$ such that
\begin{equation}\label{eq03}\mathcal J_{soc}^{(N)}(\bar u(\cdot))=\inf_{u_{i}\in\mathcal U_{i},1\leq i\leq N}\mathcal J_{soc}^{(N)}(u_1(\cdot),\cdots,u_i(\cdot),\cdots,u_{N}(\cdot)).
\end{equation}

In recent years, the large population system has been extensively discussed due to its wide applications in many areas, such as social science, engineering, economics, etc. In this structure, we should point out that each individual agent seems to be negligible, however we cannot ignore the effects of the statistical behaviors. Readers may refer to \cite{BSYS2016}, \cite{BLP2014}, \cite{CD2013}, \cite{HHN2018}, \cite{LL2007}, \cite{NNY2020}, \cite{NC2013}, \cite{SWW2021} and the references therein for MF game study.
Contrast to the aforementioned works where the agents are competitive, cooperative team optimization problem has attracted a lot of attentions in last ten years, which is the so-called \emph{social optima problem}. In \cite{HCM2012} authors investigated social optima in mean field LQG control and provided an asymptotic team-optimal solution. \cite{AM2015} focused on team-optimal control with finite population and partial information. \cite{FHH2023} investigated the homogeneity, heterogeneity and quasi-exchangeability of forward mean-field team. \cite{WZ2017} studied a mean field social optimal problem in which a Markov jump parameter appears as a common source of randomness for all agents. For more literature, one can refer to \cite{SNM18} for dynamic collective choice by finding a social optimum, \cite{NM2018} for social optima in economic models subject to idiosyncratic shocks, \cite{SSM18} for reinforcement learning algorithms for mean field teams, \cite{HN2019} for major and minor study of social optima problem, \cite{SY19} for stochastic dynamic teams and their mean field limit, \cite{WZZ2020} for uniform stabilization of mean field linear quadratic control, \cite{HWY2021} for volatility uncertainty problem, etc. For more researches and applications readers can be referred to \cite{BDT2019}, \cite{CBM2016}, \cite{PRT2015}, \cite{WS2000} and the references therein.

It is well known that objective expectation $\mathbb E$ partially represents people's preference. Alternatively, we apply the so-called \emph{generalized expectation} which seems to be subjective in some sense. Based on the theory of backward stochastic differential equation (BSDE), \cite{Peng1997} introduced the so-called $g$-expectation (nonlinear expectation), which is denoted by $\mathcal{E}_g$. Replacing $\mathbb E$ by $\mathcal{E}_g$ in \eqref{eq01}, we obtain an extension of \textbf{Problem 0}, which may be considered as nonlinear preferences. In details, let $g:[0,T]\times \RR^{1+N}\to\RR$ be a given map,
$\eta_i\in L_{\mathcal F_T}^2(\Omega;\mathbb R)$, and $Y_i(\cdot)$ satisfy the BSDE
\begin{equation}\label{eq04}\begin{aligned}
dY_i(t)=\ &g(t,Y_i(t),Z_{i\cdot}(t))dt+Z_{i\cdot}(t)dW(t),\quad Y_i(T)=\ \eta_i,\quad 1\leq i\leq N.
\end{aligned}\end{equation}
Define $\mathcal{E}_g[\eta_i]:=Y_i(0;\eta_i),\  \eta_i\in L_{\mathcal F_T}^2(\Omega;\mathbb R),\  1\leq i\leq N.$
$\mathcal{E}_g[\eta_i]$ is called the $g$-expectation of $\eta_i$, $1\leq i\leq N$, if $g(t,y,0,\cdots,0)=0$ holds for $(t,y)\in [0,T]\times \mathbb R,\ a.s.$ (\cite{Peng1997, Yong2008}). In fact, $g$-expectation is related to \emph{stochastic differential utility} introduced in \cite{DE92}. According to \cite{DE92}, we regard $Y_i(\cdot)$ as the \emph{stochastic differential utility process} of $\eta_i$, $1\leq i\leq N$ with the so-called aggregator $g(\cdot)$. In this case, the operator $\mathcal{E}_g:L_{\mathcal F_T}^2(\Omega)\rightarrow\mathbb R$ posseses most properties of $\mathbb E$ except the linearity. By virtue of the above theory, we replace \eqref{eq01} by
\begin{equation}\label{eq05}\begin{aligned}
\mathcal J_g^i(x_{i0},u(\cdot))=\ &\mathcal E_g\Bigg\{\int_0^TL(t,X_i(t),u_i(t),X^{(N)}(t))dt+\Phi(x_{i0},X_i(T))\Bigg\}=Y_i(0),
\end{aligned}\end{equation}
with $(Y_i(t),Z_{i\cdot}(t))$ being the unique adapted solution of BSDE
\begin{equation}\left\{\label{eq06}\begin{aligned}
dY_i(t)=\ &g\big(t,Y_i(t),Z_{i\cdot}(t)\big)dt+Z_{i\cdot}(t)dW(t),\\
Y_i(T)=\ &\int_0^TL(t,X_i(t),u_i(t),X^{(N)}(t))dt+\Phi(x_{i0},X_i(T)),\quad 1\leq i\leq N.
\end{aligned}\right.\end{equation}
If we define $X^*_i(t):=\int_0^t L(s,X_i(s),u_i(s),X^{(N)}(s))ds$, then \textbf{Problem 0} becomes the one to minimize $$\mathcal J_{soc}^{(N)*}(u(\cdot))=\sum_{i=1}^N\mathcal J_g^i(x_{i0},u(\cdot)),$$ where
$$\mathcal J_g^i(x_{i0},u(\cdot))=Y_i(0)=\mathcal E_g\Big\{X^*_i(T)+\Phi(x_{i0},X_i(T))\Big\}$$
subject to
\begin{equation}\left\{\label{eq07}\begin{aligned}
&d\left(
   \begin{smallmatrix}
     X^*_i(t) \\
     X_i(t) \\
   \end{smallmatrix}
 \right)=\left(
   \begin{smallmatrix}
     L(t,X_i(t),u_i(t),X^{(N)}(t))\\
     b(t,X_i(t),u_i(t),X^{(N)}(t)) \\
   \end{smallmatrix}
 \right)dt+\left(
   \begin{smallmatrix}
     0\\
    \sigma(t,X_i(t),u_i(t),X^{(N)}(t)) \\
   \end{smallmatrix}
 \right)dW_i(t),\\
&dY_i(t)=g(t,Y_i(t),Z_{i\cdot}(t))dt+Z_{i\cdot}(t)dW(t),\\
&\left(
   \begin{smallmatrix}
     X^*_i(0) \\
     X_i(0) \\
   \end{smallmatrix}
 \right)=\left(
   \begin{smallmatrix}
     0 \\
     x_{i0} \\
   \end{smallmatrix}
 \right),\quad Y_i(T)=X^*_i(T)+\Phi(x_{i0},X_i(T)),\quad 1\leq i\leq N.
\end{aligned}\right.\end{equation}
Indeed, \eqref{eq07} is a controlled FBSDE with large population structure.

The next motivation  involves a recursive utility problem. Assume that a market contains $N$ participants. The dynamic $x_i(\cdot)$ of the  individual underlying state (asset) for $i$th participant is given by
\begin{equation*}\left\{\begin{aligned}
dX_i(t)&=\left(A X_i(t)+B \pi_i(t)+FX^{(N)}(t)\right)dt+D \pi_i(t) dW_i(t),\\
X_i(0)&= x_{i0}>0,\quad 1\leq i\leq N.
\end{aligned}\right.
\end{equation*}
Here, $X^{(N)}(\cdot)=\frac{1}{N}\sum_{i=1}^NX_i(\cdot)$ is the average asset; $A,B,F,D,x_{i0}$ are constants; $(W_i(\cdot),\ 1\leq i\leq N)$ is a $N$-dimensional standard Brownian motion; $\pi_i(\cdot)\in \mathbb{R}$ is regarded as some economic indicator such as the investment strategy of the $i^{th}$ participant, $1\leq i\leq N$.

Let $c_i(\cdot),$ $1\leq i\leq N$ be a continuous consumption rate process. Assume that a terminal reward $\Phi X_i(T)$ is involved. By \cite{EPQ1997}, the recursive utility operates as a solution of a BSDE, which is denoted by $Y_i^{c_i,\pi_i}(\cdot),\ 1\leq i\leq N$. Suppose that $Y_i^{c_i,\pi_i}(\cdot),\ 1\leq i\leq N$ satisfies
\begin{equation*}\left\{\begin{aligned}
-dY_i(t)&=\left(H Y_i(t)+K c_i(t)+MX^{(N)}(t)\right)dt-Z_{i\cdot}(t) dW(t),\\
Y_i(T)&=\Phi X_i(T).
\end{aligned}\right.
\end{equation*}
%If we regard $(x,y)$ as a $2N$-dim vector, solving the matrix-valued system, it is easy to see that the BSDE admits a unique solution $(y_i^{c_i,\pi_i}(\cdot),z_{ij}^{c_i,\pi_i}(\cdot)),\ i=1,\cdots, N,\ j=1,\cdots, N$.
Define $\mathcal F_t:=\sigma \{W_i(s); 0\leq s \leq t,\ 1\leq i \leq N\}$. To find an $\mathcal F_t$-adapted process $(\bar{c}_i(\cdot),\bar{\pi}_i(\cdot))$ such that $$ J_{soc}^{(N)}(\bar{c}_i,\bar{\pi}_i):=\sum_{i=1}^NY_i^{\bar{c}_i,\bar{\pi}_i}(0)=\max_{(c_i,\pi_i)}\Bigg(\sum_{i=1}^NY_i^{c_i,\pi_i}(0)\Bigg),$$
is identified as a recursive optimal control problem.

It should be noticed that the above models are illustrated by FBSDE,
which has been extensively discussed in literature. Readers are referred to \cite{CVM}, \cite{CWZ}, \cite{MY1999}, \cite{PengWu99}, \cite{WWX2013}, \cite{WWX2015}, \cite{Yong2008}, \cite{YZ1999} and the references therein for backgrounds and applications of FBSDE. For backward LQ problems, one may refer to \cite{HWW2016}, \cite{LSX17}, \cite{LZ2001}, etc.

\subsection{Problem formulation}\label{formulation}

Motivated by the above problems, with consideration to obtain some explicit results, in this paper we study an LQ large population system in which $K$ types of heterogeneous agents $\{\mathcal A_i:1\leq i\leq N\}$ are involved, where the dynamics of the agents satisfy a class of linear FBSDEs with MF coupling:
that is, for $1\leq i\leq N,$
\begin{equation}\left\{\label{eq1}\begin{aligned}
dX_i(t)=\ &\Big[A_{\theta_i}(t)X_i(t)+B(t)u_i(t)+F(t)X^{(N)}(t)\Big]dt+\Big[D(t)u_i(t)+\sigma_i(t)\Big]dW_i(t),\\
X_i(0)=\ &\xi_i,
\end{aligned}\right.\end{equation} and
\begin{equation}\left\{\label{eq2}\begin{aligned}
dY_i(t)=\ &-\Big[H_{\theta_i}(t)Y_i(t)+K(t)u_i(t)+L(t)X_i(t)+M(t)X^{(N)}(t)\Big]dt+Z_{i\cdot}(t)dW(t),\\
Y_i(T)=\ &\Phi X_i(T)+\eta_i,
\end{aligned}\right.\end{equation}
where $X^{(N)}(\cdot)=\frac{1}{N}\sum_{i=1}^NX_i(\cdot)$ stands for the forward state-average of the agents. In this system, ``heterogeneous" means that the agents in different types are not identical statistically. In the MF social optima problem of this paper, for $1\leq i\leq N$, $(Y_i(\cdot),Z_{ij}(\cdot),1\leq j\leq N)\in L^2_{\FF}(\Omega;C([0,T];\mathbb R^m)) \times L^{2}_{\FF}(0, T; (\mathbb{R}^{m})^N)$ is called the \emph{solution} of BSDE \eqref{eq2}. We should point out that $(Z_{ij}(\cdot),1\leq j\leq N)$ is a part of the \emph{solution} introduced to make $Y_i(\cdot)$ satisfy the adaptation requirement. The coefficients $(A_{\theta_i}(\cdot),B(\cdot),F(\cdot),D(\cdot),$ $H_{\theta_i}(\cdot),K(\cdot),L(\cdot),M(\cdot),\sigma_i(\cdot)),\ 1\leq i\leq N$ depend on time variable $t$, which  is often suppressed if no confusion is caused. $\Phi$ is an $m\times n$ matrix; $\xi_i$ and $\eta_i,\ 1\leq i\leq N$ are random variables. $\theta_i$ is a number, standing for a dynamic parameter relevant to $\mathcal A_i,\ 1\leq i\leq N$ which is used to illustrate the heterogeneous feature. For notational simplicity, here we assume that only $A(\cdot)$ and $H(\cdot)$ depend on $\theta_i,\ 1\leq i\leq N$. If other parameters also depend on $\theta_i,\ 1\leq i\leq N$, corresponding analysis is similar, and thus we will not give the details. We also assume that $\theta_i,\ 1\leq i\leq N$ take values in the finite set $\Theta$ defined as $\Theta:=\{1,2,\cdots,K\}$. $\mathcal A_i$ is called a \emph{type-$k$ agent} if $\theta_i=k\in\Theta,\ 1\leq i\leq N$. %In this paper, we are interested in the asymptotic behavior as $N$ tends to infinity.
For $1 \leq k \leq K$ and a given $N$, define $\mathcal{I}_k:=\{i|\theta_i=k, 1 \leq i \leq  N\}, \  N_k:=|\mathcal{I}_k|,$ where $|\mathcal{I}_k|$ is the cardinality of the index set $\mathcal{I}_k$. For $1\leq k\leq K$, we define $\pi_k^{(N)}:=\frac{N_k}{N}$. Then $\pi^{(N)}=(\pi_1^{(N)}, \cdots, \pi_K^{(N)})$ is a probability vector representing the empirical distribution of $\theta_1, \cdots, \theta_N.$ Introduce the following assumption:
\begin{description}
  \item[(A1)] There exists a probability mass vector $\pi=(\pi_1, \cdots, \pi_K)$ such that $\displaystyle{\lim_{N\rightarrow\infty}}\pi^{(N)}=\pi$, $\displaystyle{\min_{1 \leq k \leq K}}\pi_{k}>0.$
  \item[(A2)] For $1\leq i\leq N$, $\xi_i\in L_{\mathcal F_0^i}^\infty(\Omega;\mathbb R^{n})$, $\eta_i\in L_{\mathcal F_T^i}^\infty(\Omega;\mathbb R^{m})$. $\xi_i$ and $\xi_j$ (\emph{resp.} $\eta_i$ and $\eta_j$)
   are identically distributed if $\theta_i=\theta_j=k$, and this type-$k$ variable is typically denoted by $\xi^{(k)}$ (\emph{resp.} $\eta^{(k)}$) when only their
   distribution is concerned. Here $\cF^i_t$ is the filtration gererated by the Brownian motion $W_i$.
   \item[(A3)] $A_{\theta_i}(\cdot),F(\cdot)\in L^\infty(0,T;\mathbb R^{n\times n})$, $B(\cdot),D(\cdot)\in L^\infty(0,T;\mathbb R^{n\times d})$, $H_{\theta_i}(\cdot)\in L^\infty(0,T;\mathbb R^{m\times m})$,\\ $L(\cdot),M(\cdot)\in L^\infty(0,T;\mathbb R^{m\times n})$, $K(\cdot)\in L^\infty(0,T;\mathbb R^{m\times d})$, $\sigma_i(\cdot)\in L^\infty(0,T;\mathbb R^{n})$, $1\leq i\leq N$.
  \end{description}
For $1\leq i\leq N$, the centralized admissible strategy set for $\mathcal A_i$ is defined by
$\mathcal U_i^c=\Big\{u_i(\cdot)|u_i(\cdot)\in L^2_{\FF}(0,T;\mathbb R^d)\Big\}.$
Correspondingly, the decentralized one for $\mathcal A_i$ is given by
$\mathcal U_i^d=\Big\{u_i(\cdot)|u_i(\cdot)\in L^2_{\FF^i}(0,T;\mathbb R^d)\Big\}$,\ $1\leq i\leq N.$
Actually, $\mathcal U_i^d$ is a subset of $\mathcal U_i^c,\ 1\leq i\leq N$.
It follows from (A1)-(A3) that \eqref{eq1}-\eqref{eq2} admits a unique solution for all $u_i \in \mathcal{U}_i^c,\ 1\leq i\leq N.$ In fact, we can rewrite \eqref{eq1}-\eqref{eq2} as a high-dimensional FBSDE and derive the wellposedness by the classical theory of FBSDE.  %In fact, if we denote by

Denote by $u=(u_1,\cdots,u_N)$, $u_{-i}=(u_1,\cdots,$ $u_{i-1},u_{i+1},\cdots,u_N)$, $1\leq i\leq N$. The cost functional of $\mathcal A_i$, $1\leq i\leq N$ is
\begin{equation}\label{eq3}\begin{aligned}
\mathcal J_i(u_i(\cdot),u_{-i}(\cdot))=\ &\frac{1}{2}\mathbb E\Bigg\{\int_0^T\Big[\left\langle Q(t)\left(X_i(t)-S(t) X^{(N)}(t)\right),X_i(t)-S (t) X^{(N)}(t)\right\rangle\\
&\qquad\qquad+\left\langle R_{\theta_i}(t)u_i(t),u_i(t)\right\rangle\Big]dt+\left\langle \Gamma Y_i(0),Y_i(0)\right\rangle\Bigg\}.
\end{aligned}\end{equation}
The aggregated team functional of $N$ agents is
\begin{equation}\label{eq4}
\mathcal J_{soc}^{(N)}(u(\cdot))=\sum_{i=1}^N\mathcal J_i(u_i(\cdot),u_{-i}(\cdot)).
\end{equation}
We impose an assumption on the coefficients of \eqref{eq3}.
\begin{description}
\item[(A4)] $Q(\cdot)\in L^\infty(0,T;\mathbb S^n)$, $Q(\cdot)\geq0$, $S(\cdot)\in L^\infty(0,T;\mathbb R^{n\times n})$, $R_{\theta_i}(\cdot)\in L^\infty(0,T;\mathbb S^d)$,\\ $R_{\theta_i}(\cdot)\gg 0$, $\Gamma \in \mathbb S^{m}$, $\Gamma\geq0$, $1\leq i\leq N$.
\end{description}
Corresponding to \eqref{eq1}-\eqref{eq4}, we will propose a \emph{forward-backward stochastic LQ MF social optima problem} below. It should be noticed that the expressions, such as ``social optima", ``cooperative team optimization", ``optimal team problem", etc, illustrate similar meanings of cooperative optimization problem. For the sake of uniformity, hereafter we adopt ``social optima" to express relevant meaning.  \\

\textbf{Problem 1.} Find a strategy set $\bar u=(\bar u_1,\cdots,\bar u_N)$ where $\bar u_i(\cdot)\in \mathcal U_i^c$, $1\leq i\leq N$ such that
\begin{equation}\label{eq5}
\mathcal J_{soc}^{(N)}(\bar u(\cdot))=\inf_{u_{i}\in\mathcal U_{i}^c,1\leq i\leq N}\mathcal J_{soc}^{(N)}(u_1(\cdot),\cdots,u_i(\cdot),\cdots,u_{N}(\cdot)).
\end{equation}
If $\Gamma=0$ (\emph{resp.} $L(\cdot)\equiv0,\ M(\cdot)\equiv0,\ \Phi=0,\ Q(\cdot)\equiv0$), it degenerates to \emph{(forward) stochastic LQ MF social optima problem} (\emph{resp.} \emph{backward stochastic LQ optimal control problem}).

\begin{definition}
A strategy $\widetilde u_i(\cdot)\in\mathcal U_i^d$, $1\leq i\leq N$ is an $\varepsilon$-social decentralized optimal strategy if there exists $\varepsilon=\varepsilon(N)>0$,
$\lim_{N\rightarrow\infty}\varepsilon(N)=0$ such that
$$\frac{1}{N}\Big(\mathcal J_{soc}^{(N)}
(\widetilde u(\cdot))-\inf_{u_{i}(\cdot)\in\mathcal U_{i}^c,1\leq i\leq N}\mathcal J_{soc}^{(N)}(u(\cdot))\Big)\leq\varepsilon.$$
\end{definition}
%\begin{remark}
%In reality, the linear FBSDE in \eqref{eq1}-\eqref{eq2} stands for the dynamics of some investment behaviors such as stocks and bonds in a self-financed market, that is, there is no infusion or withdrawal of funds over $[0,T]$. In recursive or hedging
%problems (finance, optimal control, etc.), the FBSDE dynamics have been deeply studied in the existing literature. Now, we take the recursive problem as an example, which also motivates us to study the MF social optima problem in FBSDE setting.
%\end{remark}
\begin{remark}
Notice that $Y_i(\cdot)$, $1\leq i\leq N$ is $\FF$-adapted because $X^{(N)}(\cdot)$ is involved in the dynamic. Therefore in \eqref{eq2} $Z_{ij}(\cdot)$ appears to represent the information of $\mathcal A_i$ associated with $W_j(\cdot),1\leq j\leq N$. It refers to the characteristic of backward (forward-backward) stochastic optimal control problem.
\end{remark}

\begin{remark}
In \eqref{eq2}-\eqref{eq3}, for $1\leq i\leq N$, $Z_{ij}(\cdot),1\leq j\leq N$ do not appear in the generator of the backward dynamic and the cost functional. It is because if $Z_{i\cdot}(\cdot)$ do, we need to make the error estimation between $\sum_{i=1}^N Z_{i\cdot}(\cdot)$ and some related quantity as those in Proposition \ref{prop0914} below. However, this seems to be an impossible task based on existing BSDE theory. Similarly, it is worthy pointing out that due to the difficulties of error estimations of BSDE, $X_i(\cdot)$ does not enter into the diffusion term in \eqref{eq1}. As a future work, we hope to overcome this difficulty with the help of some new technique. Besides, if \eqref{eq3} contains the linear term of $Y_i(0)$, similar analysis can be employed, and for simplicity of writing, here we just include the quadratic term.
\end{remark}
Now we briefly present the route of study of \textbf{Problem 1}:
%\begin{figure}[!h]
%\centering
%\setlength{\abovecaptionskip}{-0.1cm}
%\hspace*{0.4cm}\includegraphics[width=6in,height=2.5in]{pic1-1.eps}
%\caption{The research route.}\label{f1}
%\end{figure}
\begin{itemize}
 \item Firstly, we focus on solving a fully-coupled FBSDE system (so-called consistency condition system) by Riccati equation analysis. %Thanks to the wellposedness of consistency condition system, we determine some frozen MF terms and introduced variables.
  \item Based on forward-backward person-by-person optimality principle and variational synthesization technique, we obtain an auxiliary LQ control problem. Stochastic maximum principle (cf. \cite{Peng93}) is applied to solve it.
 \item By virtue of standard estimations of FBSDE, we verify that the decentralized strategy is  asymptotically optimal for centralized strategy.
\end{itemize}

The rest of the paper is organized as follows. The consistency condition system and its solvability are established in Section 2. We apply forward-backward person-by-person optimality principle to derive an auxiliary LQ control problem of each agent in Section 3. In Section 4, the asymptotic optimality of decentralized strategy is obtained. Section 5 concludes this paper.

\section{Consistency condition system} \label{cc33}
\setcounter{equation}{0}
\renewcommand{\theequation}{\thesection.\arabic{equation}}
%\subsection{Mean-field FBSDE}

In this section, we do some preparatory work. We present the consistency condition system and its wellposedness, based on which some quantities related to \eqref{eq17} and \eqref{Hamil} (see below in Section 3) will be determined, and furthermore the decentralized strategy will be derived. For the sake of presentation, we let $n=m$ here. There is no essential difference if $n\neq m$.

Consider the following stochastic system: for $1\leq k\leq K$,
\begin{equation}\label{CC}\left\{\begin{aligned}
&d\alpha_k(t)=\Bigg[A_k\alpha_k-BR_k^{-1}\big(B^\top \widetilde{\beta}_k+D^\top \widetilde{\gamma}_k+K^\top \widetilde{\alpha}_k\big)+F\sum_{l=1}^K\pi_l\mathbb E\alpha_l \Bigg]dt\\
&\qquad\qquad+\Big[-DR_k^{-1}\big(B^\top \widetilde{\beta}_k+D^\top \widetilde{\gamma}_k+K^\top \widetilde{\alpha}_k\big)+\sigma_k(t)\Big]dW^{(k)}(t),\\
&d\beta_k(t)=-\Bigg[H_k\beta_k-KR_k^{-1}\big(B^\top \widetilde{\beta}_k+D^\top \widetilde{\gamma}_k+K^\top \widetilde{\alpha}_k\big)+L\alpha_k+M\sum_{l=1}^K\pi_l\mathbb E\alpha_l\Bigg]dt+\gamma_kdW^{(k)}(t),\\
&d\widetilde{\beta}_k(t)=-\Bigg[A_k^\top \widetilde{\beta}_k+L^\top \widetilde{\alpha}_k+Q\alpha_k-(QS+ S^\top Q -S^\top QS)\sum_{l=1}^K\pi_l\mathbb E\alpha_l\\
&\qquad\qquad+\sum_{l=1}^K\pi_lF^\top \vartheta_l+\sum_{l=1}^K \pi_l(F^\top \mathbb E\check Y_l-M^\top\check X_l)\Bigg]dt+\widetilde{\gamma}_kdW^{(k)}(t),\\
&d\widetilde{\alpha}_k(t)=H_k^\top \widetilde{\alpha}_kdt,\\
&d\check X_k(t)=H_k^\top \check X_k dt,\\
&d\check Y_k(t)=-\Big[A_k^\top \check Y_k-L^\top \check X_k+Q\alpha_k\Big] dt+\check Z_kdW^{(k)}(t),\\
&d\vartheta_k(t)=-\Bigg[A_k^\top \vartheta_k-(QS+ S^\top Q-S^\top QS)\sum_{l=1}^K\pi_l\mathbb E\alpha_l+\sum_{l=1}^K\pi_lF^\top \vartheta_l+\sum_{l=1}^K\pi_l(F^\top \mathbb E\check Y_l-M^\top \check X_l)\Bigg]dt,\\
&\alpha_k(0)=\xi^{(k)},\quad \widetilde{\alpha}_k(0)=\Gamma \beta_k(0),\quad\check X_k(0)=-\Gamma\beta_k(0),\\
&\check Y_k(T)=-\Phi^\top \check X_k(T),\quad \vartheta_k(T)=0,\quad \beta_k(T)=\Phi \alpha_k(T)+\eta^{(k)},\quad \widetilde{\beta}_k(T)=\Phi^\top \widetilde{\alpha}_k(T).
\end{aligned}\right.\end{equation}
It should be noticed that stochastic system \eqref{CC} is a fully-coupled FBSDE which contains three forward SDEs, four BSDEs and some MF terms. From the analysis in Section 3, it represents some consistency properties and hence it is called \emph{consistency condition system}. The solvability of consistency condition is crucial for all large population and social optima problems. Without it the theoretical analysis will lose its significance and error estimations cannot be proceeded.

In the following, we pose a proposition to solve consistency condition system \eqref{CC}. Before that, we introduce some notation. Denote by $\mathbb X=(\alpha_1^\top,\cdots,\alpha_K^\top,\widetilde{\alpha}_1^\top,\cdots,\widetilde{\alpha}_K^\top,\check X_1^\top,\cdots,\check X_K^\top)^\top$, $\mathbb Y=( \beta_1^\top,\cdots,\beta_K^\top,\widetilde{\beta}_1^\top,\cdots,\widetilde{\beta}_K^\top,$\\ $\check Y_1^\top$, $\cdots,\check Y_K^\top$, $\vartheta_1^\top,\cdots,\vartheta_K^\top)^\top,\mathbb Z=( \gamma_1^\top,\cdots,\gamma_K^\top,\widetilde{\gamma}_1^\top,\cdots,\widetilde{\gamma}_K^\top,$ $\check Z_1^\top,$ $\cdots,\check Z_K^\top)^\top$, MF FBSDE \eqref{CC} then takes the form of
\begin{equation}\label{eq23}\left\{\begin{aligned}
&d\mathbb X=\Big[\mathbb A_1\mathbb X+\mathbb B_1\mathbb Y+\mathbb B_2\mathbb Z+\bar{\mathbb A}_1\mathbb E[\mathbb X]\Big]dt+\Big[\mathbb C\mathbb X+\mathbb D_1\mathbb Y+\mathbb D_2\mathbb Z+\Sigma_0\Big]\circ d\mathbb{W}(t),\\
&d\mathbb Y=-\Big[\mathbb A_2 \mathbb Y+\mathbb A_3\mathbb X+\mathbb B_3 \mathbb Z+\bar{\mathbb A}_2\mathbb E[\mathbb Y]+\bar{\mathbb A}_3\mathbb E[\mathbb X]\Big]dt+\left(\begin{matrix} \mathbb Z\\ \mathbf{0} \end{matrix}\right)\circ \left(\begin{matrix} d\mathbb{W}(t)\\ \mathbf{0} \end{matrix}\right),\\
&\mathbb X(0)=\Xi+\bar \Gamma \mathbb Y(0),\quad \mathbb Y(T)=\bar\Phi \mathbb X(T)+\Sigma,
\end{aligned}\right.\end{equation}
where
\tiny\begin{equation*}\begin{aligned}
&\mathbb A_1=\left(
               \begin{smallmatrix}
                 A_1 &  & &-BR_1^{-1}K^\top & & &0 & & \\
                   & \ddots &   &   & \ddots &   &   &  \ddots &  \\
                   &   & A_K &   &   & -BR_K^{-1}K^\top &   &   & 0  \\
                  0 &   &   & H_1^\top &   &   & 0  &   &   \\
                   & \ddots &   &   & \ddots &   &   & \ddots  &   \\
                   &   & 0  &   &   & H_K^\top &   &   & 0  \\
                  0&  &   & 0  &  &   & H_1^\top &   &   \\
                   & \ddots  &   &   & \ddots  &   &   & \ddots &   \\
                   &  & 0  &   &   & 0  &   &   & H_K^\top \\
               \end{smallmatrix}
             \right),   \mathbb B_1=\left(
               \begin{smallmatrix}
                 0 &  & &-BR_1^{-1}B^\top & & &0 & & &0& & \\
                   & \ddots &   &   & \ddots &   &   &\ddots & & &\ddots &  \\
                   &   & 0 &   &   & -BR_K^{-1}B^\top &   &   &0& & &0   \\
                 0 &   &   & 0 &   &   & 0  &   & &0 & &  \\
                   & \ddots &   &   & \ddots&  &   &  \ddots &  & & \ddots&  \\
                   &   & 0  &   &   & 0  &   & & 0 & &  & 0  \\
                   0 &   &   & 0 &   &   & 0  &   & &0 & &  \\
                   & \ddots &   &   & \ddots&  &   &  \ddots &  & & \ddots&  \\
                   &   & 0  &   &   & 0  &   & & 0 & &  & 0  \\
                   \end{smallmatrix}
             \right), \\
\end{aligned}\end{equation*}
\tiny\begin{equation*}\begin{aligned}
& \mathbb B_2=\left(
               \begin{smallmatrix}
                 0 &  & &-BR_1^{-1}D^\top & & &0 & &  \\
                   & \ddots &   &   & \ddots &   &   &\ddots &    \\
                   &   & 0 &   &   & -BR_K^{-1}D^\top &   &   &0  \\
                 0 &   &   & 0 &   &   & 0  &   &  \\
                   & \ddots &   &   & \ddots&  &   &  \ddots &   \\
                   &   & 0  &   &   & 0  &   & & 0  \\
                   0 &   &   & 0 &   &   & 0  &   &  \\
                   & \ddots &   &   & \ddots&  &   &  \ddots &   \\
                   &   & 0  &   &   & 0  &   & & 0  \\
                   \end{smallmatrix}
             \right), \mathbb D_1=\left(
               \begin{smallmatrix}
                 0 &  & &-DR_1^{-1}B^\top & & &0 & & &0& & \\
                   & \ddots &   &   & \ddots &   &   &\ddots & & &\ddots &  \\
                   &   & 0 &   &   & -DR_K^{-1}B^\top &   &   &0& & &0   \\
                 0 &   &   & 0 &   &   & 0  &   & &0 & &  \\
                   & \ddots &   &   & \ddots&  &   &  \ddots &  & & \ddots&  \\
                   &   & 0  &   &   & 0  &   & & 0 & &  & 0  \\
                   0 &   &   & 0 &   &   & 0  &   & &0 & &  \\
                   & \ddots &   &   & \ddots&  &   &  \ddots &  & & \ddots&  \\
                   &   & 0  &   &   & 0  &   & & 0 & &  & 0  \\
                   \end{smallmatrix}
             \right),\\
&\mathbb C=\left(
               \begin{smallmatrix}
                 0 &  & &-DR_1^{-1}K^\top & & &0 & & \\
                   & \ddots &   &   & \ddots &   &   &  \ddots &  \\
                   &   & 0 &   &   & -DR_K^{-1}K^\top &   &   & 0  \\
                  0 &   &   & 0 &   &   & 0  &   &   \\
                   & \ddots &   &   & \ddots &   &   & \ddots  &   \\
                   &   & 0  &   &   & 0 &   &   & 0  \\
                  0&  &   & 0  &  &   & 0 &   &   \\
                   & \ddots  &   &   & \ddots  &   &   & \ddots &   \\
                   &  & 0  &   &   & 0  &   &   & 0 \\
               \end{smallmatrix}
             \right), \mathbb A_3=\left(
               \begin{smallmatrix}
                 L &  & &-KR_1^{-1}K^\top & & &0 & & \\
                   & \ddots &   &   & \ddots &   &   &  \ddots &  \\
                   &   & L &   &   & -KR_K^{-1}K^\top &   &   & 0 \\
                  Q &   &   & L^\top &   &   & -M^\top\pi_1 & \cdots & -M^\top\pi_K \\
                   & \ddots &   &   & \ddots &  &\vdots & &\vdots  \\
                   &   & Q  &   &   & L^\top &   -M^\top\pi_1 & \cdots & -M^\top\pi_K  \\
                  Q&  &   & 0  &  &   & -L^\top & &  \\
                   & \ddots  &   &   & \ddots  & & &\ddots&  \\
                   &  & Q  &   &   & 0  & & & -L^\top\\
                  0 &   &   & 0 &   &   & -M^\top\pi_1 & \cdots & -M^\top\pi_K \\
                   & \ddots &   &   & \ddots  & &\vdots & &\vdots  \\
                   &   & 0  &   &   &  0 & -M^\top\pi_1 & \cdots & -M^\top\pi_K  \\
               \end{smallmatrix}
             \right),
 \\
&\mathbb D_2=\left(
               \begin{smallmatrix}
                 0 &  & &-DR_1^{-1}D^\top & & &0 & &  \\
                   & \ddots &   &   & \ddots &   &   &\ddots &   \\
                   &   & 0 &   &   & -DR_K^{-1}D^\top &   &   &0&   \\
                 0 &   &   & 0 &   &   & 0  &   &  \\
                   & \ddots &   &   & \ddots&  &   &  \ddots &   \\
                   &   & 0  &   &   & 0  &   & & 0   \\
                   0 &   &   & 0 &   &   & 0  &   &   \\
                   & \ddots &   &   & \ddots&  &   &  \ddots &   \\
                   &   & 0  &   &   & 0  &   & & 0  \\
                   \end{smallmatrix}
             \right),\bar{\mathbb A}_1=\left(
                         \begin{smallmatrix}
                           F\pi_1 & \cdots & F\pi_K & 0 &  &  &0 &  &    \\
                           \vdots &   & \vdots &  & \ddots&  &  & \ddots&     \\
                           F\pi_1 & \cdots & F\pi_K &   &  &0 &  & &0     \\
                           0  &   &   & 0 &   &   & 0 &   &     \\
                             & \ddots  &   &  & \ddots  &   & & \ddots  &     \\
                             &   & 0  & &   & 0  & &   &  0   \\
                         \end{smallmatrix}
                       \right),
\mathbb{W}=\left(
       \begin{smallmatrix}
         W^{(1)} \\
         \vdots \\
         W^{(K)} \\
         W^{(1)} \\
         \vdots \\
         W^{(K)} \\
         W^{(1)} \\
         \vdots \\
         W^{(K)} \\
       \end{smallmatrix}
     \right),
  \\
&\mathbb A_2=\left(
               \begin{smallmatrix}
                 H_1 &  & &-KR_1^{-1}B^\top & & &0 & & & 0& & \\
                   & \ddots &   &   & \ddots &   &   &  \ddots & & &\ddots&  \\
                   &   & H_K &   &   & -KR_K^{-1}B^\top &   &   & 0 & & & 0 \\
                  0 &   &   & A_1^\top &   &   & 0  &   &  & F^\top\pi_1 & \cdots & F^\top\pi_K \\
                   & \ddots &   &   & \ddots &   &   & \ddots  & &\vdots & &\vdots  \\
                   &   & 0  &   &   & A_K^\top &   &   & 0 & F^\top\pi_1 & \cdots & F^\top\pi_K  \\
                  0&  &   & 0  &  &   & A_1^\top &   & & 0& &  \\
                   & \ddots  &   &   & \ddots  &   &   & \ddots & & &\ddots&  \\
                   &  & 0  &   &   & 0  &   &   & A_K^\top & & & 0\\
                  0 &   &   & 0 &   &   & 0  &   &  & A_1^\top+F^\top\pi_1 & \cdots & F^\top\pi_K \\
                   & \ddots &   &   & \ddots &   &   & \ddots  & &\vdots & &\vdots  \\
                   &   & 0  &   &   & 0 &   &   & 0 & F^\top\pi_1 & \cdots & A_K^\top+F^\top\pi_K  \\
               \end{smallmatrix}
             \right),\\
&\bar{\mathbb A}_2=\left(
               \begin{smallmatrix}
                 0 &  & &0 & & &0 & & & 0& & \\
                   & \ddots &   &   & \ddots &   &   &  \ddots & & &\ddots&  \\
                   &   & 0 &   &   & 0 &   &   & 0 & & & 0 \\
                  0 &   &   & 0 &   &   & F^\top \pi_1 & \cdots  &F^\top \pi_K & 0  &   &\\
                   & \ddots &   &   & \ddots &   &  \vdots &   & \vdots& &\ddots &  \\
                   &   & 0  &   &   & 0 & F^\top \pi_1  &  \cdots & F^\top \pi_K &  &  & 0  \\
                  0&  &   & 0  &  &   & 0 &   & & 0& &  \\
                   & \ddots  &   &   & \ddots  &   &   & \ddots & & &\ddots&  \\
                   &  & 0  &   &   & 0  &   &   & 0 & & & 0\\
                   0 &   &   & 0 &   &   & F^\top \pi_1 & \cdots  &F^\top \pi_K & 0  &   &\\
                   & \ddots &   &   & \ddots &   &  \vdots &   & \vdots& &\ddots &  \\
                   &   & 0  &   &   & 0 & F^\top \pi_1  &  \cdots & F^\top \pi_K &  &  & 0  \\
               \end{smallmatrix}
             \right),
\bar{\mathbb{A}}_3=\left(
               \begin{smallmatrix}
                 M\pi_1 & \cdots & M\pi_K &0 & & &0 & & \\
                  \vdots &  & \vdots  &   & \ddots &   &   &  \ddots &  \\
                  M\pi_1 & \cdots  & M\pi_K &   &   & 0 &   &   & 0 \\
                  -(QS+S^\top Q-S^\top QS)\pi_1 & \cdots  & -(QS+S^\top Q-S^\top QS)\pi_K  & 0 &   &   & 0 &  &  \\
                  \vdots &  & \vdots  &   & \ddots &  &  & \ddots&   \\
                  -(QS+S^\top Q-S^\top QS)\pi_1 & \cdots  & -(QS+S^\top Q-S^\top QS)\pi_K  &   &   & 0 &    &  & 0  \\
                   0&  &   & 0  &  &   & 0 & &  \\
                  & \ddots  &   &   & \ddots  & & &\ddots&  \\
                   &  & 0  &   &   & 0  & & & 0\\
                    -(QS+S^\top Q-S^\top QS)\pi_1 & \cdots  & -(QS+S^\top Q-S^\top QS)\pi_K  & 0 &   &   & 0 &   &   \\
                  \vdots &   & \vdots  &   & \ddots  & &  & \ddots&   \\
                   -(QS+S^\top Q-S^\top QS)\pi_1 & \cdots  & -(QS+S^\top Q-S^\top QS)\pi_K  &   &   &  0 &   &   & 0  \\
               \end{smallmatrix}
             \right),\\
             & \mathbb B_3=\left(
               \begin{smallmatrix}
                 0 &  & &-KR_1^{-1}D^\top & & &0 & &  \\
                   & \ddots &   &   & \ddots &   &   &  \ddots &  \\
                   &   & 0 &   &   & -KR_K^{-1}D^\top &   &   & 0 \\
                  0 &   &   & 0 &   &   & 0  &   &   \\
                   & \ddots &   &   & \ddots &   &   & \ddots  &   \\
                   &   & 0  &   &   & 0 &   &   & 0   \\
                  0&  &   & 0  &  &   & 0 &   &   \\
                   & \ddots  &   &   & \ddots  &   &   & \ddots &   \\
                   &  & 0  &   &   & 0  &   &   & 0 \\
                  0 &   &   & 0 &   &   & 0  &   &    \\
                   & \ddots &   &   & \ddots &   &   & \ddots  &    \\
                   &   & 0  &   &   & 0 &   &   & 0  \\
               \end{smallmatrix}
             \right),\bar{\Phi}=\left(
               \begin{smallmatrix}
                 \Phi &   &   &0 & & &0 & & \\
                    &  \ddots&   &   & \ddots &   &   &  \ddots &  \\
                    &    & \Phi &   &   & 0 &   &   & 0 \\
                  0 &    &    & \Phi^\top &   &   & 0 &  &  \\
                    & \ddots &    &   & \ddots &  &  & \ddots&   \\
                    &    & 0  &   &   & \Phi^\top &    &  & 0  \\
                   0&  &   & 0  &  &   & -\Phi^\top & &  \\
                  & \ddots  &   &   & \ddots  & & &\ddots&  \\
                   &  & 0  &   &   & 0  & & & -\Phi^\top\\
                    0 &    &    & 0 &   &   & 0 &   &   \\
                    &  \ddots &    &   & \ddots  & &  & \ddots&   \\
                     &    & 0  &   &   &  0 &   &   & 0  \\
               \end{smallmatrix}
             \right),\Sigma=\left( \begin{smallmatrix}
                                                            \eta^{(1)} \\
                                                            \vdots \\
                                                            \eta^{(K)} \\
                                                            0 \\
                                                            \vdots \\
                                                            0 \\
                                                            0 \\
                                                            \vdots \\
                                                            0 \\
                                                            0 \\
                                                            \vdots \\
                                                            0 \\
                                                          \end{smallmatrix}
                                                        \right),
                                                        \\
\end{aligned}\end{equation*}
\tiny\begin{equation*}\begin{aligned}
&\bar{\Gamma}=\left(
               \begin{smallmatrix}
                 0 &  & &0 & & &0 & & &0& & \\
                   & \ddots &   &   & \ddots &   &   &\ddots & & &\ddots &  \\
                   &   & 0 &   &   & 0 &   &   &0& & &0   \\
                 \Gamma &   &   & 0 &   &   & 0  &   & &0 & &  \\
                   & \ddots &   &   & \ddots&  &   &  \ddots &  & & \ddots&  \\
                   &   & \Gamma  &   &   & 0  &   & & 0 & &  & 0  \\
                   -\Gamma &   &   & 0 &   &   & 0  &   & &0 & &  \\
                   & \ddots &   &   & \ddots&  &   &  \ddots &  & & \ddots&  \\
                   &   & -\Gamma  &   &   & 0  &   & & 0 & &  & 0  \\
                   \end{smallmatrix}
             \right),\Sigma_0=\left(
       \begin{smallmatrix}
         \sigma_1 \\
         \vdots \\
         \sigma_K \\
         0 \\
         \vdots \\
         0 \\
         0 \\
         \vdots \\
         0 \\
       \end{smallmatrix}
     \right),\Xi=\left(
       \begin{smallmatrix}
         \xi^{(1)} \\
         \vdots \\
         \xi^{(K)} \\
         0 \\
         \vdots \\
         0 \\
         0 \\
         \vdots \\
         0 \\
       \end{smallmatrix}
     \right).
\end{aligned}\end{equation*}
\normalsize
Here, ``$\circ$" denotes the \emph{generalized} Hadamard product. It is well known that Hadamard product (also called Schur product or entry-wise product) is a binary operation between two matrices of the same dimensions, and it  produces another matrix in which each element $(i,j)$ is the product of the elements $(i,j)$ in the original matrices. In this part we formally express the Hadamard product (called the \emph{generalized} Hadamard product), though dimension of the dynamic is $n$, which is different from that of the Brownian motion (1-dimension). %The above system is highly-augmented due to the coupling in the heterogenous agents.
%In the following, we will use the Riccati equation theory to discuss the well-posedness of system \eqref{eq23}, which is shown as a coupled FBSDE.
\begin{proposition}\label{phi}
Under \emph{(}A1\emph{)}-\emph{(}A4\emph{)}, assume that
\begin{equation}\label{eq24}\left\{\begin{aligned}
&\dot{\phi}+\phi\big(\mathcal{A}_1+\hat\Gamma\mathcal{A}_3\big)+\big(\mathcal{A}_2+\mathcal{A}_3\hat \Gamma\big)\phi+\phi\big(\mathcal{A}_1\hat\Gamma+\hat\Gamma\mathcal{A}_3\hat\Gamma+\hat\Gamma\mathcal{A}_2+\mathcal{B}_1\big)\phi+\mathcal{A}_3\\
&\qquad-\big[\phi\big(\mathcal{B}_2+\hat\Gamma\mathcal{B}_3\big)+\mathcal{B}_3\big]\big[\phi\big(\mathcal{D}_2-\hat\Gamma\hat I\big)-\hat I\big]^{-1}\big[\phi\mathcal{C}+\phi\big(\mathcal{C}\hat\Gamma+\mathcal{D}_1\big)\phi\big]=0,\\
& \phi(T)= \bar I\hat\Phi,
\end{aligned}\right.\end{equation}
where
\begin{equation*}\begin{aligned}
&\mathcal{A}_j=\left(
                 \begin{array}{cc}
                   \mathbb A_j +\bar{\mathbb A}_j & \mathbf{0} \\
                   \mathbf{0} & \mathbb A_j \\
                 \end{array}
               \right), \mathcal{B}_j=\left(
                 \begin{array}{cc}
                   \mathbb B_j  & \mathbf{0} \\
                   \mathbf{0} & \mathbb B_j \\
                 \end{array}
               \right),\mathcal{C}=\left(
                           \begin{array}{cc}
                             \mathbf{0} & \mathbf{0} \\
                             \mathbb{C} & \mathbb{C} \\
                           \end{array}
                         \right),\mathcal{D}_l=\left(
                           \begin{array}{cc}
                             \mathbf{0} & \mathbf{0} \\
                             \mathbb{D}_l & \mathbb{D}_l \\
                           \end{array}
                         \right),\\
& \hat{I}=\left(
                           \begin{array}{cc}
                             \mathbf{0} & \mathbf{0} \\
                             I & I \\
                             \mathbf{0}  &  \mathbf{0}
                           \end{array}
                         \right)_{8Kn\times 6Kn},\bar I=\left(
                                                                              \begin{array}{cc}
                                                                                I-\bar \Phi\bar \Gamma & \mathbf{0} \\
                                                                                \mathbf{0} & I \\
                                                                              \end{array}
                                                                            \right)^{-1},\hat{\Gamma}=\left(
                           \begin{array}{cc}
                            \bar\Gamma & \mathbf{0} \\
                             \mathbf{0} & \mathbf{0} \\
                           \end{array}
                         \right),\hat{\Phi}=\left(
                           \begin{array}{cc}
                            \bar\Phi & \mathbf{0} \\
                             \mathbf{0} &  \bar\Phi \\
                           \end{array}
                         \right),\ j=1,2,3,\ l=1,2
\end{aligned}\end{equation*}
 admits a unique solution $\phi(\cdot)$ over $[0,T]$ such that $\phi (\mathcal{D}_2-\hat \Gamma\hat I)-\hat I$ is invertible. Then, consistency condition system \eqref{CC} has a solution.
\end{proposition}
\begin{proof} Taking the expectation on both sides of \eqref{eq23}, we have
\begin{equation}\label{eq25}\left\{\begin{aligned}
&d\mathbb E[\mathbb X]=\Big[(\mathbb A_1+\bar{\mathbb A}_1)\mathbb E[\mathbb X]+\mathbb B_1\mathbb E[\mathbb Y]+\mathbb B_2\mathbb E[\mathbb Z]\Big]dt,\\
&d\mathbb E[\mathbb Y]=-\Big[(\mathbb A_2 +\bar{\mathbb A}_2)\mathbb E[\mathbb Y]+(\mathbb A_3 +\bar{\mathbb A}_3)\mathbb E[\mathbb X]+ \mathbb B_3\mathbb E[\mathbb Z]\Big]dt,\\
&\mathbb E[\mathbb X](0)=\mathbb E[\Xi]+\bar\Gamma\mathbb E[\mathbb Y](0),\quad \mathbb E[\mathbb Y](T)=\bar\Phi\mathbb E[\mathbb X](T)+\mathbb E[\Sigma].
\end{aligned}\right.\end{equation}
%It follows from \eqref{eq23} and \eqref{eq25} that
%\begin{equation}\label{eq26}\left\{\begin{aligned}
%&d\Big(\mathbb X-\mathbb E[\mathbb X]\Big)=\Big[\mathbb A_1\Big(\mathbb X-\mathbb E[\mathbb X]\Big)+\mathbb B_1\Big(\mathbb Y-\mathbb E[\mathbb Y]\Big)+\mathbb B_2\Big(\mathbb Z-\mathbb E[\mathbb Z]\Big)\Big]dt\\
%&\qquad\qquad\qquad\qquad+\Big[\mathbb C\mathbb X+\mathbb D_1\mathbb Y+\mathbb D_2\mathbb Z+\Sigma_0\Big]\circ d\mathbb{W}(t),\\
%&d\Big(\mathbb Y-\mathbb E[\mathbb Y]\Big)=-\Big[\mathbb A_2 \Big(\mathbb Y-\mathbb E[\mathbb Y]\Big)+\mathbb A_3\Big(\mathbb X-\mathbb E[\mathbb X]\Big)+\mathbb B_3\Big(\mathbb Z-\mathbb E[\mathbb Z]\Big)\Big]dt+\left(\begin{matrix} \mathbb Z\\ \mathbf{0} \end{matrix}\right)\circ \left(\begin{matrix} d\mathbb{W}(t)\\ \mathbf{0} \end{matrix}\right),\\
%&\mathbb X(0)-\mathbb E[\mathbb X](0)=\Xi-\mathbb E[\Xi],\ \mathbb Y(T)-\mathbb E[\mathbb Y](T)=\bar\Phi\Big(\mathbb X(T)-\mathbb E[\mathbb X](T)\Big)+\Sigma-\mathbb E[\Sigma].
%\end{aligned}\right.\end{equation}
Denote \begin{equation*}\begin{aligned}\mathcal{X}=\left(
                       \begin{array}{c}
                         \mathbb E[\mathbb X] \\
                         \mathbb X-\mathbb E[\mathbb X] \\
                       \end{array}
                     \right),\qquad\mathcal{Y}=\left(
                       \begin{array}{c}
                         \mathbb E[\mathbb Y] \\
                         \mathbb Y-\mathbb E[\mathbb Y] \\
                       \end{array}
                     \right),\qquad\mathcal{Z}=\left(
                       \begin{array}{c}
                         \mathbb E[\mathbb Z] \\
                         \mathbb Z-\mathbb E[\mathbb Z] \\
                       \end{array}
                     \right),
\end{aligned}\end{equation*}
$$\mathcal{W}=\left(
                       \begin{array}{c}
                         \mathbb{W} \\
                         \mathbb{W} \\
                       \end{array}
                     \right),\qquad \hat{\Xi}=\left(
                       \begin{array}{c}
                         \mathbb E[\Xi] \\
                         \Xi-\mathbb E[\Xi] \\
                       \end{array}
                     \right),\qquad\hat{\Sigma}=\left(
                       \begin{array}{c}
                         \mathbb E[\Sigma] \\
                         \Sigma-\mathbb E[\Sigma] \\
                       \end{array}
                     \right),\qquad\hat{\Sigma}_0=\left(
                       \begin{array}{c}
                         \mathbf{0} \\
                         \Sigma_0 \\
                       \end{array}
                     \right).$$
Thus MF FBSDE \eqref{eq23} is equivalent to
\begin{equation}\label{eq27}\left\{\begin{aligned}
&d\mathcal{X}=\Big[\mathcal{A}_1 \mathcal{X}+\mathcal{B}_1 \mathcal{Y}+\mathcal{B}_2 \mathcal{Z}\Big]dt+\Big[\mathcal{C} \mathcal{X}+\mathcal{D}_1 \mathcal{Y}+\mathcal{D}_2 \mathcal{Z}+\hat{\Sigma}_0\Big]\circ d\mathcal{W}(t),\\
&d\mathcal{Y}=-\Big[\mathcal{A}_2 \mathcal{Y}+\mathcal{A}_3\mathcal{X}+\mathcal{B}_3\mathcal{Z} \Big]dt+\hat I\big(\mathcal{Z}\circ d\mathcal{W}(t)\big),\\
&\mathcal{X}(0)=\hat{\Xi} +\hat \Gamma\mathcal{Y}(0),\qquad  \mathcal{Y}(T)=\hat\Phi\mathcal{X}(T)+\hat\Sigma.
\end{aligned}\right.\end{equation}
Define $\tilde{\mathcal{X}}(t)=\mathcal{X}(t)-\hat\Gamma\mathcal{Y}(t)-\hat\Xi, \  t\in [0,T].$ Then
$\mathcal{X}(0)=\hat{\Xi} +\hat \Gamma\mathcal{Y}(0)$ implies $\tilde{\mathcal{X}}(0)= 0$. By \eqref{eq27} and $d\tilde{\mathcal{X}}=d\mathcal{X}-\hat \Gamma d\mathcal{Y},$ we have
\begin{equation}\label{eq28}\left\{\begin{aligned}
&d\tilde{\mathcal{X}}=\Big[(\mathcal{A}_1+\hat\Gamma\mathcal{A}_3)\tilde{\mathcal{X}}+(\mathcal{A}_1\hat\Gamma+\hat\Gamma\mathcal{A}_3\hat\Gamma+\hat\Gamma\mathcal{A}_2+\mathcal{B}_1)\mathcal{Y}+(\mathcal{B}_2+\hat\Gamma\mathcal{B}_3)\mathcal{Z}\\
&\qquad+(\mathcal{A}_1+\hat\Gamma\mathcal{A}_3)\hat\Xi\Big]dt+\Big[\mathcal{C}\tilde{\mathcal{X}}+(\mathcal{C}\hat\Gamma+\mathcal{D}_1)\mathcal{Y}+(\mathcal{D}_2-\hat\Gamma\hat I)\mathcal{Z}+\mathcal{C}\hat\Xi+\hat{\Sigma}_0\Big]\circ d\mathcal{W}(t),\\
&d\mathcal{Y}=-\Big[(\mathcal{A}_2+\mathcal{A}_3\hat\Gamma) \mathcal{Y}+\mathcal{A}_3\tilde{\mathcal{X}}+\mathcal{B}_3\mathcal{Z}+\mathcal{A}_3\hat\Xi\Big]dt+\hat I\big(\mathcal{Z}\circ d\mathcal{W}(t)\big),\\
&\tilde{\mathcal{X}}(0)=0,\qquad  \mathcal{Y}(T)=\bar I\hat\Phi\tilde{\mathcal{X}}(T)+\bar I(\hat\Phi\hat\Xi+\hat\Sigma),
\end{aligned}\right.\end{equation}
which is a standard fully-coupled FBSDE. Here, $\bar I$ is defined by $\left(
                                                                              \begin{smallmatrix}
                                                                                I-\bar \Phi\bar \Gamma & \mathbf{0} \\
                                                                                \mathbf{0} & I \\
                                                                              \end{smallmatrix}
                                                                            \right)^{-1}
$. %$I$ is the identity matrix.
In fact, we can easily obtain that $\left(
                                                                              \begin{smallmatrix}
                                                                                I-\bar \Phi\bar \Gamma & \mathbf{0} \\
                                                                                \mathbf{0} & I \\
                                                                              \end{smallmatrix}
                                                                            \right)$ is a lower triangular matrix and the diagonal elements are all one. Thus it is invertible.
Assume that $\tilde{\mathcal{X}}$ and $\mathcal{Y}$ have the following relationship $$\mathcal Y(t)=\phi(t) \tilde{\mathcal X}(t)+\psi(t), \qquad t\in [0,T],$$
where $\phi:[0,T]\rightarrow \mathbb{R}^{8Kn\times 6Kn}$ is a deterministic matrix-valued function and $\psi:[0,T]\times \Omega\rightarrow \mathbb{R}^{8Kn}$ is an $\{\mathcal F_t\}_{t\geq 0}$-adapted process. Now we will derive $\phi(\cdot)$ and $\psi(\cdot)$. The terminal values of $\tilde{\mathcal{X}}$ and $\mathcal{Y}$ imply that $\phi(T)=\bar I\hat \Phi,\ \ \psi(T)= \bar I(\hat\Phi\hat\Xi+\hat\Sigma).$ Since $\hat\Xi\in L_{\mathcal F_0^W}^2(\Omega;\mathbb R^{6Kn})$, $\hat\Sigma\in L_{\mathcal F_T^W}^2(\Omega;\mathbb R^{8Kn})$, and $\psi(\cdot)$ is required to be $\{\mathcal F_t^W\}_{t\geq 0}$-adapted, we suppose that $\psi(\cdot)$ satisfies a BSDE
\begin{equation}\label{eq29}\left\{\begin{aligned}
&d\psi(t)=a(t)dt+\tilde I\big(b(t) \circ d\mathcal{W}(t)\big),\\
&\psi(T)=\bar I(\hat\Phi\hat\Xi+\hat\Sigma),
\end{aligned}\right.\end{equation}
where $(a(\cdot),b(\cdot))\in L^2_{\mathcal F^W}(0,T;\mathbb R^{8Kn})\times L^2_{\mathcal F^W}(0,T;\mathbb R^{6Kn})$ is undetermined; $\tilde I$ is the $8Kn\times 6Kn$-dimensional matrix, in which the elements are all 1. Here, the given matrix $\tilde I$ plays a role in increasing the dimension of $b(t) \circ d\mathcal{W}(t)$ to coincide with $\psi(t)$.
Applying It\^{o}'s formula to $\phi(t) \tilde{\mathcal X}(t)+\psi(t)$ and comparing the coefficients with the second equation in \eqref{eq28}, we get
%\begin{equation*}\begin{aligned}
%d\mathcal Y =&\ \dot{\phi} \tilde{\mathcal X}dt+\phi d\tilde{\mathcal X}+adt+\tilde I\big(b(t) \circ d\mathcal{W}(t)\big)\\
% =&\ \Big\{\Big[\dot{\phi}+\phi\big(\mathcal{A}_1+\hat\Gamma\mathcal{A}_3\big)+\phi\big(\mathcal{A}_1\hat\Gamma+\hat\Gamma\mathcal{A}_3\hat\Gamma+\hat\Gamma\mathcal{A}_2+\mathcal{B}_1\big)\phi\Big]\tilde{\mathcal X}+\phi\big(\mathcal{B}_2+\hat\Gamma\mathcal{B}_3\big)\mathcal{Z}\\
%&  \ +\phi\big(\mathcal{A}_1\hat\Gamma+\hat\Gamma\mathcal{A}_3\hat\Gamma+\hat\Gamma\mathcal{A}_2+\mathcal{B}_1\big)\psi+\phi\big(\mathcal{A}_1+\hat\Gamma\mathcal{A}_3\big)\hat\Xi+a\Big\}dt+\phi\Big(\Big[\big(\mathcal{C}+\big(\mathcal{C}\hat\Gamma+\mathcal{D}_1\big)\phi\big)\\
%&\ \cdot\tilde{\mathcal X}  +\big(\mathcal{D}_2-\hat\Gamma\hat I\big)\mathcal{Z}+\big(\mathcal{C}\hat\Gamma+\mathcal{D}_1\big)\psi+\mathcal{C}\hat\Xi+\hat{\Sigma}_0\Big]\circ d\mathcal{W}(t)\Big)+\tilde I\big(b(t) \circ d\mathcal{W}(t)\big).
%\end{aligned}\end{equation*}
%Comparing it with the second equation in \eqref{eq28}, we obtain that
\begin{equation*}\begin{aligned}
&\Big[\dot{\phi}+\phi\big(\mathcal{A}_1+\hat\Gamma\mathcal{A}_3\big)+\big(\mathcal{A}_2+\mathcal{A}_3\hat \Gamma\big)\phi+\phi\big(\mathcal{A}_1\hat\Gamma+\hat\Gamma\mathcal{A}_3\hat\Gamma+\hat\Gamma\mathcal{A}_2+\mathcal{B}_1\big)\phi+\mathcal{A}_3\Big]\tilde{\mathcal X}\\
&\quad +\phi\big(\mathcal{A}_1\hat\Gamma+\hat\Gamma\mathcal{A}_3\hat\Gamma+\hat\Gamma\mathcal{A}_2+\mathcal{B}_1\big)\psi+\big(\mathcal{A}_2+\mathcal{A}_3\hat\Gamma\big)\psi+\phi\big(\mathcal{B}_2+\hat\Gamma\mathcal{B}_3\big)\mathcal Z+\mathcal B_3\mathcal Z\\
&\quad+\phi\big(\mathcal{A}_1+\hat\Gamma\mathcal{A}_3\big)\hat\Xi+\mathcal{A}_3\hat\Xi+a=0,
\end{aligned}\end{equation*}
and
\begin{equation*}\begin{aligned}
\phi&\Big(\Big[\big(\mathcal{C}+\big(\mathcal{C}\hat\Gamma+\mathcal{D}_1\big)\phi\big)\tilde{\mathcal X}  +\big(\mathcal{D}_2-\hat\Gamma\hat I\big)\mathcal{Z}+\big(\mathcal{C}\hat\Gamma+\mathcal{D}_1\big)\psi+\mathcal{C}\hat\Xi+\hat{\Sigma}_0\Big]\circ d\mathcal{W}(t)\Big)\\
&+\tilde I\big(b(t) \circ d\mathcal{W}(t)\big)=\hat I\big(\mathcal{Z}\circ d\mathcal{W}(t)\big).
\end{aligned}\end{equation*}
By some matrix calculations, we derive
$$\Big[\phi\mathcal{C}+\phi\big(\mathcal{C}\hat\Gamma+\mathcal{D}_1\big)\phi\Big]\tilde{\mathcal X}+\phi\big(\mathcal{D}_2-\hat\Gamma\hat I\big)\mathcal{Z}+\phi\big(\mathcal{C}\hat\Gamma+\mathcal{D}_1\big)\psi-\hat I\mathcal{Z}+\phi\big(\mathcal{C}\hat\Xi+\hat{\Sigma}_0\big)+\tilde Ib=0.$$
Since $\phi\big(\mathcal{D}_2-\hat\Gamma\hat I\big)-\hat I$ is invertible, it follows that $$\mathcal{Z}=-\big[\phi\big(\mathcal{D}_2-\hat\Gamma\hat I\big)-\hat I\big]^{-1}\Big[\Big(\phi\mathcal{C}+\phi\big(\mathcal{C}\hat\Gamma+\mathcal{D}_1\big)\phi\Big)\tilde{\mathcal X}+\phi\big(\mathcal{C}\hat\Gamma+\mathcal{D}_1\big)\psi+\phi\big(\mathcal{C}\hat\Xi+\hat{\Sigma}_0\big)+\tilde Ib\Big].$$
Noticing \eqref{eq24}, we have
\begin{equation*}\begin{aligned}
a=&-\Big[\phi\big(\mathcal{A}_1\hat\Gamma+\hat\Gamma\mathcal{A}_3\hat\Gamma+\hat\Gamma\mathcal{A}_2+\mathcal{B}_1\big)\psi+\big(\mathcal{A}_2+\mathcal{A}_3\hat\Gamma\big)\psi+\phi\big(\mathcal{A}_1+\hat\Gamma\mathcal{A}_3\big)\hat\Xi+\mathcal{A}_3\hat\Xi\\
&-\big[\phi\big(\mathcal{B}_2+\hat\Gamma\mathcal{B}_3\big)+\mathcal{B}_3\big]\big[\phi\big(\mathcal{D}_2-\hat\Gamma\hat I\big)-\hat I\big]^{-1}\big[\phi\big(\mathcal{C}\hat\Gamma+\mathcal{D}_1\big)\psi+\phi\big(\mathcal{C}\hat\Xi+\hat{\Sigma}_0\big)+\tilde Ib\big]\Big].
\end{aligned}\end{equation*}
Then equation \eqref{eq29} has the form of
\begin{equation}\label{eq30}\left\{\begin{aligned}
d\psi=&\ -\Big\{\Big[\mathcal{A}_2+\mathcal{A}_3\hat\Gamma+\phi\big(\mathcal{A}_1\hat\Gamma+\hat\Gamma\mathcal{A}_3\hat\Gamma+\hat\Gamma\mathcal{A}_2+\mathcal{B}_1\big)\\
&\ -\big[\phi\big(\mathcal{B}_2+\hat\Gamma\mathcal{B}_3\big)+\mathcal{B}_3\big]\big[\phi\big(\mathcal{D}_2-\hat\Gamma\hat I\big)-\hat I\big]^{-1}\phi\big(\mathcal{C}\hat\Gamma+\mathcal{D}_1\big)\Big]\psi+\big[\phi\big(\mathcal{A}_1+\hat\Gamma\mathcal{A}_3\big)\\
&\ +\mathcal{A}_3\big]\hat\Xi -\big[\phi\big(\mathcal{B}_2+\hat\Gamma\mathcal{B}_3\big)+\mathcal{B}_3\big]\big[\phi\big(\mathcal{D}_2-\hat\Gamma\hat I\big)-\hat I\big]^{-1}\big(\phi\mathcal{C}\hat\Xi+\phi\hat{\Sigma}_0+\tilde Ib\big)\Big\}dt\\
&\ +\tilde I\big(b(t) \circ d\mathcal{W}(t)\big), \qquad \psi(T)=\bar I(\hat\Phi\hat\Xi+\hat\Sigma).
\end{aligned}\right.\end{equation}

If \eqref{eq24} admits a solution $\phi(\cdot)$ such that $\phi\big(\mathcal{D}_2-\hat\Gamma\hat I\big)-\hat I$ is invertible, BSDE \eqref{eq30} admits a unique adapted solution $(\psi(\cdot),b(\cdot))$. Then the equation of $\tilde{\mathcal{X}}$ (SDE)
%\begin{equation}\label{eq31}\left\{\begin{aligned}
%&d\tilde{\mathcal{X}}=\Big\{\Big[\mathcal{A}_1+\hat\Gamma\mathcal{A}_3+(\mathcal{A}_1\hat\Gamma+\hat\Gamma\mathcal{A}_3\hat\Gamma+\hat\Gamma\mathcal{A}_2+\mathcal{B}_1)\phi\\
%&\qquad\quad-(\mathcal{B}_2+\hat\Gamma\mathcal{B}_3)\big[\phi\big(\mathcal{D}_2-\hat\Gamma\hat I\big)-\hat I\big]^{-1}\big[\phi\mathcal{C}+\phi\big(\mathcal{C}\hat\Gamma+\mathcal{D}_1\big)\phi\big]\Big]\tilde{\mathcal{X}}\\
%&\qquad\quad+(\mathcal{A}_1\hat\Gamma+\hat\Gamma\mathcal{A}_3\hat\Gamma+\hat\Gamma\mathcal{A}_2+\mathcal{B}_1)\psi+(\mathcal{A}_1+\hat\Gamma\mathcal{A}_3)\hat\Xi\\
%&\qquad\quad-(\mathcal{B}_2+\hat\Gamma\mathcal{B}_3)\big[\phi\big(\mathcal{D}_2-\hat\Gamma\hat I\big)-\hat I\big]^{-1}\big[\phi\big(\mathcal{C}\hat\Gamma+\mathcal{D}_1\big)\psi+\phi\big(\mathcal{C}\hat\Xi+\hat{\Sigma}_0\big)+\tilde Ib\big]\Big\}dt\\
%&\qquad\quad+\Big\{\Big[\mathcal{C}+(\mathcal{C}\hat\Gamma+\mathcal{D}_1)\phi-(\mathcal{D}_2-\hat\Gamma\hat I)\big[\phi\big(\mathcal{D}_2-\hat\Gamma\hat I\big)-\hat I\big]^{-1}\big[\phi\mathcal{C}+\phi\big(\mathcal{C}\hat\Gamma+\mathcal{D}_1\big)\phi\big]\Big]\tilde{\mathcal{X}}\\
%&\qquad\quad+(\mathcal{C}\hat\Gamma+\mathcal{D}_1)\psi-(\mathcal{D}_2-\hat\Gamma\hat I)\big[\phi\big(\mathcal{D}_2-\hat\Gamma\hat I\big)-\hat I\big]^{-1}\big[\phi\big(\mathcal{C}\hat\Gamma+\mathcal{D}_1\big)\psi+\phi\big(\mathcal{C}\hat\Xi+\hat{\Sigma}_0\big)\\
%&\qquad\quad+\tilde Ib\big]+\mathcal{C}\hat\Xi+\hat{\Sigma}_0\Big\}\circ d\mathcal{W}(t),\\
%&\tilde{\mathcal{X}}(0)=0
%\end{aligned}\right.\end{equation}
admits a unique solution $\tilde{\mathcal{X}}(\cdot)$. Further, $(\mathcal Y(\cdot),\mathcal Z(\cdot))$ is derived. Then $\mathcal X(\cdot)$ is obtained. The proof is complete.
\end{proof}

\begin{remark}
Notice that \eqref{eq27} is a fully-coupled FBSDE, in which $\mathcal{Y}(T)$ depends on $\mathcal{X}(T)$ and $\mathcal{X}(0)$ depends on $\mathcal{Y}(0)$. This kind of FBSDE has been studied, see e.g., \cite{LSX17, LZ2001}, etc. However, \eqref{eq27} is quite different from those of the existing works. It should be noticed that $\mathcal{A}_1+\hat\Gamma\mathcal{A}_3\neq\mathcal{A}_2+\mathcal{A}_3\hat \Gamma$ in general, which implies $\phi$ is asymmetric. It follows from the above analysis that $\mathbb{Y}$ and $\mathbb{X}$ have different dimensions, which leads to the asymmetry of $\phi$. Thus this asymmetry is brought by the characteristics of the social optimisation problem itself. Actually, it is a challenge to derive the solvability of Riccati equation \eqref{eq24}. For constant coefficient case, explicit solution may be obtained by direct calculations under additional conditions (see e.g., \cite{AFIJ, Yong06}, etc) .
\end{remark}

Now we look at a special case. We let either $\mathcal B_2=0,\ \mathcal B_3=0,$ or $\mathcal C=0,\ \mathcal D_1=0,$
which means either $B=0,\ K=0,$ or $D=0,\ H_k=0\ (k=1,\cdots,K)$.
In this case, \eqref{eq24} becomes
\begin{equation}\label{eq32}\left\{\begin{aligned}
&\dot{\phi}+\phi\big(\mathcal{A}_1+\hat\Gamma\mathcal{A}_3\big)+\big(\mathcal{A}_2+\mathcal{A}_3\hat \Gamma\big)\phi+\phi\big(\mathcal{A}_1\hat\Gamma+\hat\Gamma\mathcal{A}_3\hat\Gamma+\hat\Gamma\mathcal{A}_2+\mathcal{B}_1\big)\phi+\mathcal{A}_3=0,\\
& \phi(T)= \bar I\hat\Phi.
\end{aligned}\right.\end{equation}
\begin{proposition}\label{p433}
Let \emph{(}A1\emph{)}-\emph{(}A4\emph{)} hold. Let $(U(\cdot),V(\cdot))$ be the solution  of the ODE
\begin{equation}\nonumber\left\{\begin{aligned}
&\left(
   \begin{array}{c}
     \dot{U}(t) \\
     \dot{V}(t) \\
   \end{array}
 \right)
=\left(
   \begin{array}{cc}
     \mathcal{A}_1+\hat\Gamma\mathcal{A}_3 & \mathcal{A}_1\hat\Gamma+\hat\Gamma\mathcal{A}_3\hat\Gamma+\hat\Gamma\mathcal{A}_2+\mathcal{B}_1 \\
     -\mathcal{A}_3 & -(\mathcal{A}_2+\mathcal{A}_3\hat \Gamma) \\
   \end{array}
 \right)\left(
   \begin{array}{c}
     U(t) \\
     V(t) \\
   \end{array}
 \right),\quad  t\in[0,T],\\
&\left(
   \begin{array}{c}
     {U}(T) \\
     {V}(T) \\
   \end{array}
 \right)=\left(
   \begin{array}{c}
     I \\
     \bar I\hat\Phi \\
   \end{array}
 \right),
\end{aligned}\right.\end{equation}
and $U(\cdot)$ is nonsingular on $[0,T]$. Then $\phi(t)=V(t)U^{-1}(t)$ is the unique solution of \eqref{eq32}.
\end{proposition}
\begin{proof}
We adapt the method of \cite[Theorem 5.12]{E2005}. Differentiation of the identity $U(t)U^{-1}(t)=I$ gives
$$\dot U(t)U^{-1}(t)+U(t)\frac{d}{dt}\big\{U^{-1}(t)\big\}=0,$$
which implies
$$\frac{d}{dt}\big\{U^{-1}(t)\big\}=-U^{-1}(t)\dot U(t)U^{-1}(t).$$
Define $\phi(t)=V(t)U^{-1}(t)$. Then we obtain that $\phi(\cdot)$ satisfies
$$\dot{\phi}+\phi\big(\mathcal{A}_1+\hat\Gamma\mathcal{A}_3\big)+\big(\mathcal{A}_2+\mathcal{A}_3\hat \Gamma\big)\phi+\phi\big(\mathcal{A}_1\hat\Gamma+\hat\Gamma\mathcal{A}_3\hat\Gamma+\hat\Gamma\mathcal{A}_2+\mathcal{B}_1\big)\phi+\mathcal{A}_3=0$$
with $\phi(T)=V(T)U^{-1}(T)=\bar I\hat\Phi$. Hence the conclusion.
\end{proof}

Another result on the solvability of Riccati equation \eqref{eq32} is as follows.
\begin{proposition}\label{p43}
Let \emph{(}A1\emph{)}-\emph{(}A4\emph{)} hold. For $s\in[0,T]$, let $\Psi(\cdot,s)$ be the solution of the ODE
\begin{equation}\label{eq33}\left\{\begin{aligned}
&\frac{d}{dt}\Psi(t,s)=\widehat{\mathbf{A}}(t)\Psi(t,s),\quad  t\in[s,T],\\
&\Psi(s,s)=I,
\end{aligned}\right.\end{equation}
where
\small\begin{equation}\nonumber\begin{aligned}
&\widehat{\mathbf{A}}(\cdot)=\left(
                                  \begin{smallmatrix}
                                    \mathcal{A}_1+\hat\Gamma\mathcal{A}_3+\big(\mathcal{A}_1\hat\Gamma+\hat\Gamma\mathcal{A}_3\hat\Gamma+\hat\Gamma\mathcal{A}_2+\mathcal{B}_1\big)\bar I\hat\Phi
                                    & \mathcal{A}_1\hat\Gamma+\hat\Gamma\mathcal{A}_3\hat\Gamma+\hat\Gamma\mathcal{A}_2+\mathcal{B}_1 \\
                                    \begin{smallmatrix}
                                    -\big[\bar I\hat\Phi\big(\mathcal{A}_1+\hat\Gamma\mathcal{A}_3\big)\\
                                    +\big(\mathcal{A}_2+\mathcal{A}_3\hat \Gamma\big)\bar I\hat\Phi+\bar I\hat\Phi\big(\mathcal{A}_1\hat\Gamma+\hat\Gamma\mathcal{A}_3\hat\Gamma+\hat\Gamma\mathcal{A}_2+\mathcal{B}_1\big)\bar I\hat\Phi+\mathcal{A}_3\big]
                                    \end{smallmatrix}
                                    & -\big[\mathcal{A}_2+\mathcal{A}_3\hat \Gamma+ \bar I\hat\Phi\big(\mathcal{A}_1\hat\Gamma+\hat\Gamma\mathcal{A}_3\hat\Gamma+\hat\Gamma\mathcal{A}_2+\mathcal{B}_1\big)\big] \\
                                  \end{smallmatrix}
                                \right).
\end{aligned}\end{equation}\normalsize
Suppose that
$$\left[\left(
               \begin{array}{cc}
                 0 & I \\
               \end{array}
             \right)
\Psi(T,t)\left(
                             \begin{array}{c}
                               0 \\
                               I \\
                             \end{array}
                           \right)\right]^{-1}\in L^1(0,T;\mathbb R^{8Kn\times 8Kn}).$$
Then \eqref{eq24} admits a unique solution $\phi(\cdot)$, which is given by
\begin{equation}\label{eq34}\begin{aligned}
\phi(t)=\bar I\hat\Phi-\left[\left(
               \begin{array}{cc}
                 0 & I \\
               \end{array}
             \right)\Psi(T,t)\left(
                             \begin{array}{c}
                               0 \\
                               I \\
                             \end{array}
                           \right)\right]^{-1}\left(
               \begin{array}{cc}
                 0 & I \\
               \end{array}
             \right)\Psi(T,t)\left(
                             \begin{array}{c}
                               I \\
                               0 \\
                             \end{array}
                           \right),\quad  t\in[0,T].
\end{aligned}\end{equation}
\end{proposition}

\begin{proof}
Define $\Pi(t)=\phi(t)-\bar I\hat\Phi , \  t\in [0,T].$
$\phi(T)=\bar I\hat\Phi $ implies $\Pi(T)= 0$. By $\phi(t)=\Pi(t)+\bar I\hat\Phi$, we obtain
\begin{equation}\nonumber\left\{\begin{aligned}
\dot{\Pi}&+\Pi\Big(\mathcal{A}_1+\hat\Gamma\mathcal{A}_3+\big(\mathcal{A}_1\hat\Gamma+\hat\Gamma\mathcal{A}_3\hat\Gamma+\hat\Gamma\mathcal{A}_2+\mathcal{B}_1\big)\bar I\hat\Phi\Big) +\Big(\mathcal{A}_2+\mathcal{A}_3\hat \Gamma\\
&+\bar I\hat\Phi\big(\mathcal{A}_1\hat\Gamma+\hat\Gamma\mathcal{A}_3\hat\Gamma+\hat\Gamma\mathcal{A}_2+\mathcal{B}_1\big)\Big)\Pi+\Pi\Big(\mathcal{A}_1\hat\Gamma+\hat\Gamma\mathcal{A}_3\hat\Gamma+\hat\Gamma\mathcal{A}_2+\mathcal{B}_1 \Big)\Pi\\
&+\bar I\hat\Phi\big(\mathcal{A}_1+\hat\Gamma\mathcal{A}_3\big)+\big(\mathcal{A}_2+\mathcal{A}_3\hat \Gamma\big)\bar I\hat\Phi+\bar I\hat\Phi\big(\mathcal{A}_1\hat\Gamma+\hat\Gamma\mathcal{A}_3\hat\Gamma+\hat\Gamma\mathcal{A}_2+\mathcal{B}_1\big)\bar I\hat\Phi\\
&+\mathcal{A}_3=0,\qquad \Pi(T)=0.
\end{aligned}\right.\end{equation}
According to \cite[Theorem 5.3]{Yong06}, we have
\begin{equation}\nonumber\begin{aligned}
\Pi(t)=-\left[\left(
               \begin{array}{cc}
                 0 & I \\
               \end{array}
             \right)\Psi(T,t)\left(
                             \begin{array}{c}
                               0 \\
                               I \\
                             \end{array}
                           \right)\right]^{-1}\left(
               \begin{array}{cc}
                 0 & I \\
               \end{array}
             \right)\Psi(T,t)\left(
                             \begin{array}{c}
                               I \\
                               0 \\
                             \end{array}
                           \right),\qquad  t\in[0,T].
\end{aligned}\end{equation}
Then we get \eqref{eq34}.
\end{proof}

%For further discussing the explicit solutions of Riccati equations, %\eqref{phi-explicit} seems to be more complicated than \eqref{Phi-explicit}. We take $\phi(\cdot)$, the solution of \eqref{phi-explicit} as an example.
%we give the following proposition.
The following proposition further discusses the explicit solutions of $\phi(\cdot)$.
\begin{proposition}
Assume that $\widehat{\mathbf{A}}(\cdot)$ is a constant-valued matrix and denoted by $\widehat{\mathbf{A}}(t)\equiv\Lambda$. Suppose that
\begin{equation}\nonumber\begin{aligned}
\det\left\{\left(
               \begin{array}{cc}
                 0 & I \\
               \end{array}
             \right)e^{\Lambda t}\left(
                             \begin{array}{c}
                               0 \\
                               I \\
                             \end{array}
                           \right)
\right\}>0,\quad \forall\ t\in[0,T]
\end{aligned}\end{equation}
holds. Then \eqref{eq24} admits a unique solution $\phi(\cdot)$ as
\begin{equation}\label{eq35}\begin{aligned}
\phi(t)=\bar I\hat\Phi-\left[\left(
               \begin{array}{cc}
                 0 & I \\
               \end{array}
             \right)e^{\Lambda (T-t)}\left(
                             \begin{array}{c}
                               0 \\
                               I \\
                             \end{array}
                           \right)\right]^{-1}\left(\begin{array}{cc}
                 0 & I \\
               \end{array}\right)e^{\Lambda (T-t)}\left(
                             \begin{array}{c}
                               I \\
                               0 \\
                             \end{array}
                           \right),\qquad  t\in[0,T].
\end{aligned}\end{equation}
\end{proposition}
\begin{proof}
Similar to the proof of Proposition \ref{p43}, \eqref{eq35} is obtained by \cite[Theorem 4.3]{MY1999}.
\end{proof}

\section{Stochastic optimal control problem for $\mathcal A_i,\ 1\leq i\leq N$}\label{auxiliary problem}
\setcounter{equation}{0}
\renewcommand{\theequation}{\thesection.\arabic{equation}}

Now, we make some person-by-person analysis, introduce an optimal control problem, solve it and get the decentralized control in this section.

\subsection{Forward-backward person-by-person optimality}\label{p-b-p optimality}

It is well-known that by freezing the state-average term we can always derive an auxiliary control problem in MF game scheme. Actually, due to the difference of cost functional, MF social optima scheme and MF game scheme are rather different. In MF social optima scheme the person-by-person optimality is regarded as an effective method to derive the auxiliary control problem, see e.g. Section \ref{decentralized strategy} below or \cite{HWY2021}. In this section, we will apply variation method to analyze the MF approximation by virtue of person-by-person optimality principle. Due to forward-backward structure, we call it \emph{forward-backward person-by-person optimality}.

Let $\{\bar u_i,\bar u_{-i}\in \mathcal U_i^c\}_{i=1}^N$ denote all centralized optimal strategies. Take the perturbation into account that $\mathcal A_i$ uses $u_i\in \mathcal U_i^c$, $1\leq i\leq N$, while the other agents use $\bar u_{-i}=(\bar u_1,\cdots,\bar u_{i-1},\bar u_{i+1},\cdots,$ $\bar u_N)$. States satisfying \eqref{eq1}-\eqref{eq2} associated with $(u_i,\bar u_{-i})$ and $(\bar u_i,\bar u_{-i})$ are denoted by $(X_i,Y_i,Z_{i\cdot})$
 and $(\bar X_i,\bar Y_i,\bar Z_{i\cdot})$, respectively, $i=1,2,\cdots, N$.
 For $1\leq j\leq N$, define
$$\delta u_j=u_j-\bar u_j,\quad  \delta X_j=X_j-\bar X_j,\quad  \delta Y_j=Y_j-\bar Y_j,\quad  \delta Z_{j\cdot}=Z_{j\cdot}-\bar Z_{j\cdot}.$$
Therefore, variation of the dynamic for $\mathcal A_i$, $1\leq i\leq N$ is
\begin{equation}\label{eq6}\left\{\begin{aligned}
&d\delta X_i=\Big[A_{\theta_i}\delta X_i+B\delta u_i+F\delta X^{(N)}\Big]dt+D\delta u_idW_i(t),\\
&d\delta Y_i=-\Big[H_{\theta_i}\delta Y_i+K\delta u_i+L\delta X_i+M\delta X^{(N)}\Big]dt+\delta Z_{i\cdot}(t)dW(t),\\
&\delta X_i(0)=0,\quad\delta Y_i(T)=\Phi\delta X_i(T),
\end{aligned}\right.\end{equation}
and for $\mathcal A_j$, $j\neq i$,
\begin{equation}\label{eq9}\left\{\begin{aligned}
&d\delta X_j=\Big[A_{\theta_j}\delta X_j+F\delta X^{(N)}\Big]dt,\\
&d\delta Y_j=-\Big[H_{\theta_j}\delta Y_j+L\delta X_j+M\delta X^{(N)}\Big]dt+\delta Z_{j\cdot}(t)dW(t),\\
&\delta X_j(0)=0,\quad\delta Y_j(T)=\Phi\delta X_j(T).
\end{aligned}\right.\end{equation}
For $1\leq k\leq K$, we define $\delta X_{(k)}=\sum_{j\in\mathcal I_k,j\neq i}\delta X_j$ and $\delta Y_{(k)}=\sum_{j\in\mathcal I_k,j\neq i}\delta Y_j$. We then have
\begin{equation*}\begin{aligned}
&d\delta X_{(k)}=\Big[A_{k}\delta X_{(k)}+\left(N_k-I_{\mathcal I_k}(i)\right)F\delta X^{(N)}\Big]dt,\ \ \delta X_{(k)}(0)=0,
\end{aligned}\end{equation*} and
\begin{equation*}\begin{aligned}
&d\delta Y_{(k)}=-\Big[H_{k}\delta Y_{(k)}+L\delta X_{(k)}+\left(N_k-I_{\mathcal I_k}(i)\right)M\delta X^{(N)}\Big]dt+\sum_{j\in\mathcal I_k,j\neq i}\delta Z_{j\cdot}dW(t),\\ &\delta Y_{(k)}(T)=\Phi\delta X_{(k)}(T).
\end{aligned}\end{equation*}
Here, $I_{\mathcal I_k}(\cdot)$ denotes the indicative function. Define $\Delta\mathcal J_j=\mathcal J_j(u_j,\bar u_{-j})-\mathcal J_j(\bar u_j,\bar u_{-j})$. By elementary  calculations, we further derive the variation of cost functional for $\mathcal A_i$ as
\begin{equation*}\begin{aligned}
\Delta \mathcal J_{i}=\ &\mathbb E\Bigg\{\int_0^T\left[\left\langle Q\left(\bar X_i-S \bar X^{(N)}\right),\delta X_i-S \delta X^{(N)}\right\rangle+\left\langle R_{\theta_i}\bar u_i,\delta u_i\right\rangle\right]dt\\
&\qquad+\big\langle \Gamma\bar Y_i(0),\delta Y_i(0)\big\rangle\Bigg\}.
\end{aligned}\end{equation*}
For $j\neq i$, variation of cost functional for $\mathcal A_j$ is
\begin{equation*}\begin{aligned}
\Delta \mathcal J_{j}=\ &\mathbb E\Bigg\{\int_0^T \left[\left\langle Q\left(\bar X_j-S\bar X^{(N)}\right), \delta X_j-S\delta X^{(N)}\right\rangle \right]dt+\left\langle \Gamma\bar Y_j(0), \delta Y_j(0)\right\rangle\Bigg\}.
\end{aligned}\end{equation*}
Thus it follows
\begin{equation}\label{eq10}\begin{aligned}
\Delta \mathcal J_{soc}^{(N)}=\ &\mathbb E\Bigg\{\int_0^T\Bigg[\sum_{j=1}^N\left\langle Q\left(\bar X_j-S\bar X^{(N)}\right), \delta X_j-S\delta X^{(N)}\right\rangle+\left\langle R_{\theta_i}\bar u_i,\delta u_i\right\rangle\Bigg]dt\\
&\qquad\qquad+\sum_{j=1}^N\left\langle \Gamma\bar Y_j(0), \delta Y_j(0)\right\rangle\Bigg\}.
\end{aligned}\end{equation}
In the following, based on \eqref{eq10} we will obtain another representation of $\Delta \mathcal J_{soc}^{(N)}$ which is affected by $\delta X_i,\ \delta u_i,\ \delta Y_i$ and some error terms. We will further derive the decentralized auxiliary cost functional.
\begin{proposition}\label{prop0914}
The variation of $\mathcal J_{soc}^{(N)}$ has the following form
\begin{equation}\label{eq16}\begin{aligned}
&\Delta \mathcal J_{soc}^{(N)}\\
  =\ &\mathbb E\Bigg\{\int_0^T\Bigg[\langle Q\bar X_i,\delta X_i\rangle-\langle (QS+ S^\top Q -S^\top QS)\widehat X,\delta X_i\rangle+\sum_{k=1}^K\langle \pi_kF^\top Y_2^k,\delta X_i\rangle
\\
&+\sum_{k=1}^K\langle \pi_k(F^\top \mathbb E\mathbf{Y}_k-M^\top \mathbf{X}_k),\delta X_i\rangle+\langle R_{\theta_i}\bar u_i,\delta u_i\rangle\Bigg]dt +\langle \Gamma\bar Y_i(0),\delta Y_i(0)\rangle\Bigg\}+\sum_{l=1}^9\varepsilon_l,
\end{aligned}\end{equation}
where $(\mathbf X_k,\mathbf Y_k)$ stands for a type-$k$ representative of $(X_1^j,Y_1^j)$ when only the distribution is concerned. $\widehat X,\ X^{**}_k,\ X_j^*,\ Y_j^*,\ Z_{j\cdot}^*$ are the approximations of  $\bar X^{(N)},\ \delta X_{(k)},\ N_k\delta X_j,\ N_k\delta Y_j,$\\ $N_k\delta Z_{j\cdot}$, respectively. For $j\in\mathcal I_k,\ j\neq i$,
\begin{equation}\label{eq17}\left\{\begin{aligned}
&dX_1^j=H_k^\top X_1^j dt,\\
&dY_1^j=\Big[-A_k^\top Y_1^j+L^\top X_1^j-Q\bar X_j\Big] dt+Z_1^{j\cdot}dW(t),\\
&dY_2^k=\bigg[-A_k^\top Y_2^k+(QS+ S^\top Q-S^\top QS)\widehat X-\sum_{l=1}^K\pi_l(F^\top \mathbb E\mathbf{Y}_l-M^\top \mathbf{X}_l)\\
&\qquad\qquad-\sum_{l=1}^K\pi_lF^\top Y_2^l\bigg]dt,\\
&X_1^j(0)=-\Gamma\bar Y_j(0),\quad Y_1^j(T)=-\Phi^\top X_1^j(T),\quad Y_2^k(T)=0,\ \ 1\leq k\leq K,
\end{aligned}\right.\end{equation}
and
\begin{equation*}\left\{\begin{aligned}
&\varepsilon_1=\mathbb E\int_0^T\left\langle \left(QS+ S^\top Q- S^\top Q S\right)\left(\widehat X-\bar X^{(N)}\right),N\delta X^{(N)}\right\rangle dt,\\
&\varepsilon_2=\sum_{k=1}^K\mathbb E\int_0^T\left\langle \left(QS+ S^\top Q-S^\top QS\right)\widehat X,X_k^{**}-\delta X_{(k)}\right\rangle dt,\\
&\varepsilon_3=\sum_{k=1}^K\mathbb E\int_0^T\frac{1}{N_k}\sum_{j\in\mathcal I_k,j\neq i}\left\langle Q\bar X_j,N_k\delta X_j-X^*_j\right\rangle dt,\\
&\varepsilon_4=\sum_{k=1}^K\frac{1}{N_k}\sum_{j\in\mathcal I_k,j\neq i}\left\langle \Gamma\bar Y_j(0),N_k\delta Y_j(0)-Y^*_j(0)\right\rangle,\\
&\varepsilon_{5}=\sum_{k=1}^K\mathbb E\int_0^T\left\langle \pi_kM^\top \mathbf X_k-\frac{1}{N_k}\sum_{j\in\mathcal I_k,j\neq i}\pi_kM^\top X_1^j,\delta X_i\right\rangle dt,\\
&\varepsilon_{6}=\sum_{k=1}^K\mathbb E\int_0^T\left\langle -\pi_kF^\top \mathbb E\mathbf Y_k+\frac{1}{N_k}\sum_{j\in\mathcal I_k,j\neq i}\pi_kF^\top Y_1^j,\delta X_i\right\rangle dt,\\
&\varepsilon_{7}=\sum_{k=1}^K\mathbb E\int_0^T\left\langle \sum_{l=1}^K\pi_lM^\top\mathbf{X}_l-\sum_{l=1}^K\frac{\pi_l}{N_l}\sum_{j\in\mathcal I_l,j\neq i}M^\top X_1^j,X_k^{**}\right\rangle dt,\\
&\varepsilon_{8}=\sum_{k=1}^K\mathbb E\int_0^T\left\langle -\sum_{l=1}^K\pi_lF^\top\mathbb E\mathbf{Y}_l+\sum_{l=1}^K\frac{\pi_l}{N_l}\sum_{j\in\mathcal I_l,j\neq i}F^\top Y_1^j,X_k^{**}\right\rangle dt,\\
&\varepsilon_{9}=-\sum_{k=1}^K\mathbb E\int_0^T\frac{I_{\mathcal I_k}(i)}{N_k}\left\langle \pi_kF^\top Y_2^k,\delta X_i\right\rangle dt.
\end{aligned}\right.\end{equation*}
\end{proposition}

\begin{proof}
For the proof we may divide it into three steps.

Step I:
Replacing $\bar X^{(N)}$ in \eqref{eq10} by MF term $\widehat X$ which will be determined later. Specifically,
\small \begin{equation*}\begin{aligned}
\Delta \mathcal J_{soc}^{(N)} = \ &\mathbb E\Bigg\{\int_0^T\Bigg[\left\langle Q\left(\bar X_i-S \bar X^{(N)}\right),\delta X_i\right\rangle-\left\langle Q\left(\bar X_i-S \bar X^{(N)}\right),S \delta X^{(N)}\right\rangle\\
 &-\sum_{j=1,j\neq i}^N\left\langle QS\bar X^{(N)}, \delta X_j\right\rangle+\sum_{j=1,j\neq i}^N\left\langle Q\bar X_j, \delta X_j\right\rangle\\
 &-\sum_{j=1,j\neq i}^N\left\langle Q\left(\bar X_j-S\bar X^{(N)}\right), S\delta X^{(N)}\right\rangle+\left\langle R_{\theta_i}\bar u_i,\delta u_i\right\rangle \Bigg]dt+\sum_{j=1}^N\left\langle \Gamma\bar Y_j(0), \delta Y_j(0)\right\rangle\Bigg\}\\
  %=\ &\mathbb E\Bigg\{\int_0^T\Bigg[\left\langle Q\left(\bar X_i-S \widehat X\right),\delta X_i\right\rangle-\sum_{j=1,j\neq i}^N\left\langle QS\widehat X, \delta X_j\right\rangle+\sum_{j=1,j\neq i}^N\left\langle Q\bar X_j, \delta X_j\right\rangle\\
% &-\left\langle S^\top Q\left(\widehat X-S\widehat X\right), N\delta X^{(N)}\right\rangle+\left\langle R_{\theta_i}\bar u_i,\delta u_i\right\rangle \Bigg]dt+\left\langle \Gamma\bar Y_i(0),\delta Y_i(0)\right\rangle+\sum_{j=1,j\neq i}^N\left\langle \Gamma\bar Y_j(0), \delta Y_j(0)\right\rangle\Bigg\}+\varepsilon_1\\
 =\ &\mathbb E\Bigg\{\int_0^T\Bigg[\left\langle Q\bar X_i,\delta X_i\right\rangle-\left\langle \left(QS+ S^\top Q-S^\top QS\right)\widehat X,\delta X_i\right\rangle\\
 &
-\sum_{k=1}^K\left\langle \left(QS+ S^\top Q -S^\top QS\right)\widehat X,\delta X_{(k)}\right\rangle+ \sum_{k=1}^K\frac{1}{N_k}\sum_{j\in\mathcal I_k,j\neq i}\left\langle Q \bar X_j,N_k\delta X_{j}\right\rangle\\
&+\langle R_{\theta_i}\bar u_i,\delta u_i\rangle \Bigg]dt+\langle \Gamma\bar Y_i(0),\delta Y_i(0)\rangle+ \sum_{k=1}^K\frac{1}{N_k}\sum_{j\in\mathcal I_k,j\neq i}\langle \Gamma \bar Y_j(0),N_k\delta Y_{j}(0)\rangle\Bigg\}+\varepsilon_1.
\end{aligned}\end{equation*}\normalsize

Step II: For $1\leq k\leq K$, introduce $X^{**}_k$ to replace $\delta X_{(k)}$. For $j\in\mathcal I_k,\ j\neq i$, introduce $X_j^*$ to replace $N_k\delta X_j$ and $(Y_j^*,Z_j^*)$ to replace $(N_k\delta Y_j, N_k\delta Z_j)$, where
\begin{equation}\label{eq11}\left\{\begin{aligned}
&dX_k^{**}=\Bigg[A_kX_k^{**}+F\pi_k\delta X_i+F\pi_k\sum_{l=1}^KX_l^{**}\Bigg]dt,\\
&dX_j^*=\Bigg[A_{k}X_j^*+F\pi_k\delta X_i+F\pi_k\sum_{l=1}^KX_l^{**}\Bigg]dt,\\
&dY_j^*=-\Bigg[H_{k}Y_j^*+LX_j^*+M\pi_k\delta X_i+M\pi_k\sum_{l=1}^KX_l^{**}\Bigg]dt+Z_{j\cdot}^*dW(t),\\
&X_k^{**}(0)=0,\quad X_j^*(0)=0,\quad Y_j^*(T)=\Phi X_j^*(T).
\end{aligned}\right.\end{equation}
Therefore,
\begin{equation}\label{eq12}\begin{aligned}
&\Delta \mathcal J_{soc}^{(N)}\\
  =\ &\mathbb E\Bigg\{\int_0^T\Bigg[\langle Q\bar X_i,\delta X_i\rangle-\langle (QS+ S^\top Q -S^\top QS)\widehat X,\delta X_i\rangle\\
  &
-\sum_{k=1}^K\langle (QS+ S^\top Q-S^\top QS)\widehat X,X^{**}_k\rangle+ \sum_{k=1}^K\frac{1}{N_k}\sum_{j\in\mathcal I_k,j\neq i}\langle Q \bar X_j, X_j^* \rangle\\
&+\langle R_{\theta_i}\bar u_i,\delta u_i\rangle \Bigg]dt+\langle \Gamma\bar Y_i(0),\delta Y_i(0)\rangle+ \sum_{k=1}^K\frac{1}{N_k}\sum_{j\in\mathcal I_k,j\neq i}\langle \Gamma \bar Y_j(0),Y^*_{j}(0)\rangle\Bigg\}+\sum_{l=1}^4\varepsilon_l.
\end{aligned}\end{equation}

Step III: Substitute $X_j^*,\ Y_j^*$ and $X_k^{**}$ by dual method. Introduce the adjoint processes $(X_1^j,Y_1^j,Z_1^{j\cdot})$ and $Y_2^k$ of the terms $(Y_j^*,Z_{j\cdot}^*,X_j^*)$ and $X_k^{**}$, respectively, which are assumed to satisfy
\begin{equation}\nonumber\left\{\begin{aligned}
&dX_1^j=\alpha_1 dt,\qquad X_1^j(0)=-\Gamma\bar Y_j(0),\\
&dY_1^j=\alpha_2 dt+Z_1^{j\cdot}dW(t),\quad Y_1^j(T)=-\Phi^\top X_1^j(T),\quad j\in\mathcal I_k,\ j\neq i,\\
&dY_2^k=\alpha_3dt+Z_2^kdW_k,\qquad Y_2^k(T)=0,\quad 1\leq k\leq K,
\end{aligned}\right.\end{equation}
where $\al_1,\ \al_2,\ \al_3$ will be determined later.
Applying It\^{o}'s formula to $\langle X_1^j,Y_j^*\rangle$, we have
\begin{equation*}\begin{aligned}
d\langle X_1^j,Y_j^*\rangle=&\left[\left\langle X_1^j,-\left(H_{k}Y_j^*+LX_j^*+M\pi_k\delta X_i+M\pi_k\sum_{l=1}^KX_l^{**}\right)\right\rangle+\left\langle \alpha_1,Y_j^*\right\rangle\right]dt\\
&+\sum_{j=1}^N(\cdots)dW_j(t).
\end{aligned}\end{equation*}
For $j\in\mathcal I_k,\ j\neq i$, integrating from $0$ to $T$ and taking expectation, we obtain
\begin{equation}\label{eq13}\begin{aligned}
&\mathbb E\left\langle X_1^j(T),\Phi X^*_{j}(T)\right\rangle+\mathbb E\left\langle \Gamma \bar Y_j(0),Y^*_{j}(0)\right\rangle
=\ \mathbb E\left\langle X_1^j(T),Y_j^*(T)\right\rangle-\mathbb E\left\langle X_1^j(0),Y_j^*(0)\right\rangle\\
%=\ &\mathbb E\int_0^T\left[\left\langle X_1^j,-\left(H_{k}Y_j^*+LX_j^*+M\pi_k\delta X_i+M\pi_k\sum_{l=1}^KX_l^{**}\right)\right\rangle+\left\langle \alpha_1,Y_j^*\right\rangle\right]dt\\
=\ &\mathbb E\int_0^T\bigg[\left\langle\alpha_1-H_k^\top X_1^j,Y_j^*\right\rangle-\left\langle L^\top X_1^j,X_j^*\right\rangle-\sum_{l=1}^K\left\langle \pi_kM^\top X_1^j,X_l^{**}\right\rangle\\
&\qquad\qquad-\left\langle \pi_kM^\top X_1^j,\delta X_i\right\rangle\bigg]dt.
\end{aligned}\end{equation}
Similarly, we derive
\begin{equation}\label{eq14}\begin{aligned}
&-\mathbb E\left\langle \Phi^\top X_1^j(T), X^*_{j}(T)\right\rangle
=\mathbb E\left\langle Y_1^j(T),X_j^*(T)\right\rangle-\mathbb E\left\langle Y_1^j(0),X_j^*(0)\right\rangle\\
&=\ \mathbb E\int_0^T\left[\left\langle\alpha_2+A_k^\top Y_1^j,X_j^*\right\rangle+\sum_{l=1}^K\left\langle \pi_kF^\top Y_1^j,X_l^{**}\right\rangle+\left\langle \pi_kF^\top Y_1^j,\delta X_i\right\rangle\right]dt,
\end{aligned}\end{equation} and
\begin{equation}\label{eq15}\begin{aligned}
0=\ &\mathbb E\left\langle Y_2^k(T),X_k^{**}(T)\right\rangle-\mathbb E\left\langle Y_2^k(0),X_k^{**}(0)\right\rangle\\
=\ &\mathbb E\int_0^T\left[\left\langle\alpha_3+A_k^\top Y_2^k,X_k^{**}\right\rangle+\sum_{l=1}^K\left\langle \pi_kF^\top Y_2^k,X_l^{**}\right\rangle+\left\langle \pi_kF^\top Y_2^k,\delta X_i\right\rangle\right]dt.
\end{aligned}\end{equation}
Letting
\begin{equation*}\left\{\begin{aligned}
\alpha_1=\ &H_k^\top X_1^j,\\
\alpha_2=\ &-A_k^\top Y_1^j+L^\top X_1^j-Q\bar X_j,\\
\alpha_3=\ &-A_k^\top Y_2^k+(QS+ S^\top Q-S^\top QS)\widehat X-\sum_{l=1}^K\pi_lF^\top \mathbb EY_1^l+\sum_{l=1}^K\pi_lM^\top X_1^l\\
&\quad -\sum_{l=1}^K\pi_lF^\top Y_2^l,
\end{aligned}\right.\end{equation*}
and substituting \eqref{eq13}-\eqref{eq15} into \eqref{eq12}, we derive \eqref{eq16}. Notice that $(X_1^j,Y_1^j)$ are exchangeable for all $j\in\mathcal I_k,\ j\neq i$. Here, ``exchangeable" means there is no essential difference for $(X_1^j,Y_1^j)$ in the same type-$k$ in distribution sense. Thus we apply $(\mathbf X_k,\mathbf Y_k)$ to stand for $(X_1^j,Y_1^j)$ for the type-$k$ representative when the expectations are involved. The proof is complete.
\end{proof}

%$\varepsilon_{1}-\varepsilon_{9}$ are actually $o(1)$ order and the rigorous proof will be shown in Section \ref{asymptotic optimality}.
Based on the above analysis, we pose the following auxiliary cost functional with perturbation
\begin{equation}\label{eq18}\begin{aligned}
\Delta J_i
 =\ &\mathbb E\Bigg\{\int_0^T\Bigg[\langle Q\bar X_i,\delta X_i\rangle-\left\langle \left(QS+ S^\top Q -S^\top QS\right)\widehat X,\delta X_i\right\rangle+\sum_{k=1}^K\langle \pi_kF^\top Y_2^k,\delta X_i\rangle\\
 &+\sum_{k=1}^K\left\langle \pi_k(F^\top \widehat Y_k-M^\top \widehat X_k),\delta X_i\rangle +\langle R_{\theta_i}\bar u_i,\delta u_i\right\rangle\Bigg]dt +\langle \Gamma\bar Y_i(0),\delta Y_i(0)\rangle\Bigg\}.
\end{aligned}\end{equation}

\begin{remark}
It should be noticed that $X_1^j$ is deterministic, satisfying the ODE in \eqref{eq17} because its initial value $X_1^j(0)$ is deterministic. Since $\bar X_j$ is $\mathcal F_t$-adapted, $ Z_1^{jl}\ (1\leq l\leq N)$ cannot be omitted, though the terminal value $Y_1^j(T)$ is deterministic. By contrast, the drift term of $Y_2^k$ (the second BSDE in \eqref{eq17}) is deterministic (It follows from \eqref{eq0915a} that $\widehat X$ is deterministic (see below)) and the terminal value $Y_2^k(T)$ is zero, thus we derive that $Z_2^k$ should be zero, which implies $Y_2^k$ is deterministic indeed. Therefore, system \eqref{eq11} is a coupled FBSDE, and adjoint system \eqref{eq17} is made up by two ODEs and a BSDE.
\end{remark}
\begin{remark}
In \textbf{Step III}, $2N+K$ adjoint equations are introduced to ensure $\Delta \mathcal J_{soc}^{(N)}$ to break up with the dependence on $X_j^*,Y_j^*$ and $X_k^{**}$. This problem is caused by the existence of $X^{(N)}$ in state equations, that is $F(\cdot),\ M(\cdot)\neq0$. On the contrary, if $F(\cdot)\equiv 0,\ M(\cdot)\equiv 0$, then $X_j^*(\cdot)\equiv0,\ Y_j^*(\cdot)\equiv0$ and $X_k^{**}(\cdot)\equiv0$. There is no additional adjoint equation required to obtain auxiliary control problem.
\end{remark}

\subsection{Decentralized strategy}\label{decentralized strategy}

Motivated by \eqref{eq18}, we will pose an auxiliary forward-backward LQ optimal control problem. Firstly we substitute $X^{(N)}(\cdot)$ with $\widehat X(\cdot)$ in dynamics \eqref{eq1}-\eqref{eq2} and get the new dynamics in decentralized sense. Secondly, taking \eqref{eq18} as the perturbation of the auxiliary cost functional, one can also guess a quadratic cost functional in decentralized manner. Then we have\\

\textbf{Problem 2.} Minimize $J_i(u_i)$ over $u_i\in\mathcal U_i^d$ subject to
\begin{equation}\left\{\label{eq19}\begin{aligned}
dX_i(t)=\ &\Big[A_{\theta_i}(t)X_i(t)+B(t)u_i(t)+F(t)\widehat X(t)\Big]dt+\Big[D(t)u_i(t)+\sigma_i(t)\Big]dW_i(t),\\
X_i(0)=\ &\xi_i,
\end{aligned}\right.\end{equation} and
\begin{equation}\left\{\label{eq20}\begin{aligned}
dY_i(t)=\ &-\Big[H_{\theta_i}(t)Y_i(t)+K(t)u_i(t)+L(t)X_i(t)+M(t)\widehat X(t)\Big]dt+Z_i(t)dW_i(t),\\
Y_i(T)=\ &\Phi X_i(T)+\eta_i,
\end{aligned}\right.\end{equation}
where
\begin{equation}\label{eq21}\begin{aligned}
J_i(u_i)
 =\ &\frac{1}{2}\Bigg\{\mathbb E\int_0^T\Big[\langle Q X_i, X_i\rangle-2\langle \Theta, X_i\rangle+\langle R_{\theta_i} u_i, u_i\rangle\Big] dt+\langle \Gamma Y_i(0),Y_i(0)\rangle\Bigg\},
\end{aligned}\end{equation}
with
\begin{equation}\nonumber\begin{aligned}
&\Theta=(QS+ S^\top Q -S^\top QS)\widehat X-\sum_{k=1}^K\pi_kF^\top Y_2^k-\sum_{k=1}^K \pi_k(F^\top \widehat Y_k-M^\top \widehat X_k).
\end{aligned}\end{equation}
Here, $\widehat X,\widehat X_k,\widehat Y_k,Y_2^k$ can be chosen as
\begin{equation}\label{eq0915a}
\left(\widehat X,\widehat X_k, \widehat Y_k,Y_2^k\right)=\left(\sum\limits_{l=1}^K\pi_l\mathbb E\alpha_l,\check X_k,\mathbb E\check Y_k,\vartheta_k\right),
\end{equation}
where $\left(\alpha_k,\beta_k,\gamma_k,\widetilde{\alpha}_k,\widetilde{\beta}_k,\widetilde{\gamma}_k,\check X_k,\check Y_k,\check Z_k,\vartheta_k\right),\ 1\le k\le K$, with $\alpha_k,\beta_k,\widetilde{\beta}_k,\check Y_k\in L^2_{\FF^{W^{(k)}}}(\Omega;$ $C([0,T];\mathbb R^n))$, $\gamma_k, \widetilde{\gamma}_k,\check Z_k\in L^2_{\FF^{W^{(k)}}}(\Omega;\mathbb R^n)$ and $\widetilde{\alpha}_k,\vartheta_k\in L^2(0,T;\mathbb R^n)$ is the solution to
consistency condition system \eqref{CC}.

Now, we solve this problem by applying stochastic maximum principle, see e.g. \cite{Peng93}. Introduce an adjoint equation
\begin{equation*}\left\{\begin{aligned}
&dp_i(t)=-\Big[A_{\theta_i}^\top p_i+L^\top q_i+QX_i-\Theta\Big]dt+\overline{p}_idW_i(t),\\
&dq_i(t)=H_{\theta_i}^\top q_idt,\\
&p_i(T)=\Phi^\top q_i(T),\quad q_i(0)=\Gamma Y_i(0).
\end{aligned}\right.\end{equation*}
The stochastic maximum principle implies
\begin{equation*}\label{eq22}
\bar u_i(t)=-R_{\theta_i}^{-1}(t)\big(B^\top(t) p_i(t)+D^\top(t) \overline{p}_i(t)+K^\top(t) q_i(t)\big).
\end{equation*}
 The related Hamiltonian system becomes
\begin{equation}\label{Hamil}\left\{\begin{aligned}
&dX_i(t)=\Big[A_{\theta_i}(t)X_i(t)-B(t)R_{\theta_i}^{-1}(t)\big(B^\top(t) p_i(t)+D^\top(t) \overline{p}_i(t)+K^\top(t) q_i(t)\big)\\
&\qquad\qquad+F(t)\widehat X(t)\Big]dt+\Big[-D(t)R_{\theta_i}^{-1}(t)\big(B^\top(t) p_i(t)+D^\top(t) \overline{p}_i(t)\\
&\qquad\qquad+K^\top(t) q_i(t)\big)+\sigma_i(t)\Big]dW_i(t),\\
&dY_i(t)=-\Big[H_{\theta_i}(t)Y_i(t)-K(t)R_{\theta_i}^{-1}(t)\big(B^\top(t) p_i(t)+D^\top(t) \overline{p}_i(t)+K^\top(t) q_i(t)\big)\\
&\qquad\qquad+L(t)X_i(t)+M(t)\widehat X(t)\Big]dt+Z_i(t)dW_i(t),\\
&dp_i(t)=-\Big[A_{\theta_i}^\top (t)p_i(t)+L^\top (t) q_i(t)+Q(t)X_i(t)-\Theta\Big]dt+\overline{p}_i(t)dW_i(t),\\
&dq_i(t)=H_{\theta_i}^\top(t) q_i(t)dt,\\
&X_i(0)=\xi_i,\quad Y_i(T)=\Phi X_i(T)+\eta_i,\quad p_i(T)=\Phi^\top q_i(T),\quad q_i(0)=\Gamma Y_i(0).
\end{aligned}\right.\end{equation}

\begin{remark}
Similar to large population problem, in social optima scheme the state (resp. problem) corresponding to the external variable $\widehat X(\cdot)$ is always called auxiliary (or limiting) state (resp. problem); while the state (resp. problem) corresponding to weakly coupled term $X^{(N)}(\cdot)$ is always called real state (resp. problem).
\end{remark}
\begin{remark}
Thanks to Section 2, we derive the wellposedness of consistency condition system \eqref{CC}. Based on \eqref{CC} and \eqref{eq0915a}, one also obtains the wellposedness of \eqref{eq17} and \eqref{Hamil}.
\end{remark}

\section{Asymptotic $\varepsilon$-optimality}\label{asymptotic optimality}
\setcounter{equation}{0}
\renewcommand{\theequation}{\thesection.\arabic{equation}}

We start this section with the representation of social cost, which is to be applied to verify the asymptotic optimality.
\subsection{Representation of social cost}\label{Resc}

System \eqref{eq1}-\eqref{eq2} is rewritten as
\begin{equation}\label{eq37}
  \left\{
  \begin{aligned}
     &d\mathbf X=(\mathbf A_{\theta}\mathbf{X}+\mathbf Bu)dt+\sum_{i=1}^N(\mathbf D_iu+\overline{\sigma}_i)dW_i(t),\\
     &d\mathbf{Y} = -(\mathbf{H}_{\theta}\mathbf{Y} + \mathbf{K}u+\mathbf{M}\mathbf{X})dt + \mathbf Zd\mathbf{W}(t), \\
     &\mathbf X(0)= \overline{\xi},\quad \mathbf{Y}(T) = \mathbf{\Phi}\mathbf{X}(T)+\overline{\eta},\\
  \end{aligned}
  \right.
\end{equation}
where
\begin{equation*}
  \begin{aligned}
     & \mathbf{X}=
    \left(\begin{smallmatrix}
      X_1 \\
      \vdots\\
      X_N
    \end{smallmatrix}\right), \mathbf{Y}=
    \left(\begin{smallmatrix}
      Y_1 \\
      \vdots\\
      Y_N
    \end{smallmatrix}\right),
\mathbf Z=\begin{smallmatrix}
      1 \\
      \vdots \\
      i \\
      \vdots \\
      N
    \end{smallmatrix}\left(\begin{smallmatrix}
      Z_{11} & Z_{12} & \cdots & Z_{1i} &  \cdots & Z_{1N}\\
      \vdots & \vdots &   & \vdots &   & \vdots \\
     Z_{i1}&Z_{i2} & \cdots &Z_{ii} & \cdots & Z_{iN}\\
      \vdots & \vdots &   & \vdots &   & \vdots \\
      Z_{N1} & Z_{N2} &\cdots &Z_{Ni} & \cdots & Z_{NN}\\
    \end{smallmatrix}\right),\\
&\mathbf{A}_{\theta}=
    \left(\begin{smallmatrix}
      A_{\theta_1} + \frac{F}{N} & \frac{F}{N} & \cdots & \frac{F}{N} \\
      \frac{F}{N} & A_{\theta_2} + \frac{F}{N} & \cdots & \frac{F}{N} \\
      \vdots & \vdots & \ddots &\vdots\\
      \frac{F}{N} & \frac{F}{N} & \cdots & A_{\theta_N} + \frac{F}{N} \\
    \end{smallmatrix}\right),
\mathbf{H}_{\theta}=
    \left(\begin{smallmatrix}
      H_{\theta_1} & 0 & \cdots & 0 \\
      0 & H_{\theta_2} & \cdots & 0 \\
     \vdots & \vdots & \ddots &\vdots\\
      0 & 0 & \cdots & H_{\theta_N} \\
    \end{smallmatrix}\right),\mathbf D_i=\begin{smallmatrix}
      1 \\
      \vdots \\
      i \\
      \vdots \\
      N
    \end{smallmatrix}\left(\begin{smallmatrix}
       0 & \cdots & 0 &  \cdots & 0\\
       \vdots &   & \vdots &   & \vdots \\
     0 & \cdots &D & \cdots & 0\\
       \vdots &   & \vdots &   & \vdots \\
       0 &\cdots &0 & \cdots & 0\\
    \end{smallmatrix}\right),\\
    \end{aligned}
\end{equation*}
\begin{equation*}
  \begin{aligned}
&\mathbf{M}=
    \left(\begin{smallmatrix}
      L + \frac{M}{N} & \frac{M}{N} & \cdots & \frac{M}{N} \\
      \frac{M}{N} & L + \frac{M}{N} & \cdots & \frac{M}{N} \\
      \vdots & \vdots & \ddots &\vdots\\
      \frac{M}{N} & \frac{M}{N} & \cdots & L + \frac{M}{N} \\
    \end{smallmatrix}\right),\overline{\sigma}_i=\begin{smallmatrix}
      1 \\
      \vdots \\
      i \\
      \vdots \\
      N
    \end{smallmatrix}\left(\begin{smallmatrix}
       0 \\
       \vdots \\
     \sigma_i \\
      \vdots \\
     0\\
    \end{smallmatrix}\right),\overline{\xi} = \left(\begin{smallmatrix}
      \xi_1\\
      \vdots\\
      \xi_N
    \end{smallmatrix}\right),
\overline{\eta}= \left(\begin{smallmatrix}
      \eta_1\\
      \vdots\\
      \eta_N
    \end{smallmatrix}\right),\\
&\mathbf{B}=
    \left(\begin{smallmatrix}
      B & 0 & \cdots & 0 \\
      0 & B & \cdots & 0 \\
     \vdots & \vdots & \ddots &\vdots\\
      0 & 0 & \cdots & B \\
    \end{smallmatrix}\right),
\mathbf{K}=
    \left(\begin{smallmatrix}
      K & 0 & \cdots & 0 \\
      0 & K & \cdots & 0 \\
     \vdots & \vdots & \ddots &\vdots\\
      0 & 0 & \cdots & K \\
    \end{smallmatrix}\right),
\mathbf{\Phi}=
    \left(\begin{smallmatrix}
      \Phi & 0 & \cdots & 0 \\
      0 & \Phi & \cdots & 0 \\
     \vdots & \vdots & \ddots &\vdots\\
      0 & 0 & \cdots & \Phi \\
    \end{smallmatrix}\right),
u=
    \left(\begin{smallmatrix}
      u_1 \\
      \vdots\\
      u_N
    \end{smallmatrix}\right),\mathbf{W} = \left(\begin{smallmatrix}
      W_1\\
      \vdots\\
     W_N
    \end{smallmatrix}\right).
    \end{aligned}
\end{equation*}
Similarly, the social cost takes the form of
\begin{equation}\nonumber
  \begin{aligned}
  \mathcal{J}^{(N)}_{soc}(u)
  & =\frac{1}{2}\sum_{i=1}^N\mathbb E\Bigg\{\int_0^T\left[\big\langle Q(t)(X_i(t)-S(t) X^{(N)}(t)),X_i(t)-S (t) X^{(N)}(t)\big\rangle\right.\\
&\qquad\qquad\left.+\big\langle R_{\theta_i}(t)u_i(t),u_i(t)\big\rangle\right]dt+\langle \Gamma Y_i(0),Y_i(0)\rangle\Bigg\}\\
  & = \frac{1}{2}\Bigg\{\mathbb{E}\int_{0}^{T}\Big[ \langle \mathbf{Q}\mathbf{X},\mathbf{X}\rangle +\langle \mathbf{R}u,u\rangle \Big]dt+\langle \mathbf \Gamma\mathbf{Y}(0),\mathbf{Y}(0)\rangle\Bigg\}, \\
  \end{aligned}
\end{equation}
where
\begin{equation*}
  \begin{aligned}
     & \mathbf{Q}=
    \left(\begin{smallmatrix}
 Q + \frac{1}{N}(S^\top QS -QS - S^\top Q) &   \frac{1}{N}(S^\top QS -QS - S^\top Q) &  \cdots & \frac{1}{N}(S^\top QS -QS - S^\top Q) \\
  \frac{1}{N}(S^\top QS -QS - S^\top Q) & Q + \frac{1}{N}(S^\top QS -QS - S^\top Q) & \cdots & \frac{1}{N}(S^\top QS -QS - S^\top Q) \\
   \vdots & \vdots & \ddots &\vdots\\
 \frac{1}{N}(S^\top QS -QS - S^\top Q) & \frac{1}{N}(S^\top QS -QS - S^\top Q) & \cdots & Q + \frac{1}{N}(S^\top QS -QS - S^\top Q) \\
    \end{smallmatrix}\right) ,                                                                                      \\
     &  \mathbf{R} = \left(\begin{smallmatrix}
      R_{\theta_1} & 0 & \cdots & 0\\
      0 & R_{\theta_2} & \cdots & 0\\
      \vdots & \vdots & \ddots & \vdots\\
      0 & 0 & \cdots & R_{\theta_N}\\
    \end{smallmatrix}\right), \mathbf{\Gamma}=
    \left(\begin{smallmatrix}
 \Gamma&   0 &  \cdots & 0 \\
  0 & \Gamma & \cdots & 0\\
   \vdots & \vdots & \ddots &\vdots\\
 0 & 0 & \cdots & \Gamma \\
    \end{smallmatrix}\right).
  \end{aligned}
\end{equation*}
Let $\Psi_1(\cdot),\ \Psi_2(\cdot)$ be the solutions of
\begin{equation}\nonumber\left\{
\begin{aligned}
&d\Psi_1(t)=\mathbf A_{\theta}(t) \Psi_1(t)dt,\quad t\in[0,T],\ \  \Psi_1(0)=I,\\
&d\Psi_2(s)=\mathbf H_{\theta}(s) \Psi_2(s)ds,\quad s\in[t,T],\ \ \Psi_2(t)=I.
\end{aligned}\right.
\end{equation}
Then the strong solution $(\mathbf X,\ \mathbf Y)$ of \eqref{eq37} admits
\begin{equation}\label{eq38} \left\{
\begin{aligned}
&\mathbf X(t)=\Psi_1(t)\overline{\xi}+\Psi_1(t)\int_0^t\Psi_1(s)^{-1}\mathbf B(s)u(s)ds\\
&\qquad\qquad+\sum_{i=1}^N\Psi_1(t)\int_0^t\Psi_1(s)^{-1}\big(\mathbf D_iu(s)+\overline{\sigma}_i\big)dW_i(s),\\
&\mathbf Y(t)=\mathbb E\left[\big(\mathbf{\Phi}\mathbf{X}(T)+\overline{\eta}\big)\Psi_2(T)+\int_t^T\big(\mathbf{K}(s)u(s)+\mathbf{M}(s)\mathbf{X}(s)\big)\Psi_2(s)ds\Big|\mathcal F_t\right].
\end{aligned}\right.
\end{equation}
Define eight operators by
\begin{equation}\nonumber \left\{
\begin{aligned}
&(\mathcal L_1u(\cdot))(\cdot):=\Psi_1(\cdot)\Bigg\{\int_0^{\cdot}\Psi_1(s)^{-1}\mathbf B(s)u(s)ds+\sum_{i=1}^N\int_0^{\cdot}\Psi_1(s)^{-1}\mathbf D_iu(s)dW_i(s)\Bigg\},\\ &\widetilde{\mathcal L}_1u(\cdot):=(\mathcal L_1u(\cdot))(T),\\
&\mathcal L_2(\overline{\xi})(\cdot):=\Psi_1(\cdot)\overline{\xi}+\sum_{i=1}^N\Psi_1(\cdot)\int_0^{\cdot}\Psi_1(s)^{-1}\overline{\sigma}_idW_i(s),\quad \widetilde{\mathcal L}_2(\overline{\xi}):=\mathcal L_2(\overline{\xi})(T),\\
&(\mathcal L_3u(\cdot))(\cdot):=\mathbb E\Bigg[\int_{\cdot}^T\Bigg(\mathbf{K}(s)u(s)+\mathbf{M}(s)\Psi_1(s)\Bigg\{\int_0^{s}\Psi_1(s)^{-1}\mathbf B(s)u(s)ds\\
&\qquad\qquad\qquad+\sum_{i=1}^N\int_0^{s}\Psi_1(s)^{-1}\mathbf D_iu(s)dW_i(s)\Bigg\}\Bigg)\Psi_2(s)ds\Big|\mathcal F_{\cdot}\Bigg]\\
&\qquad\qquad\qquad+\mathbb E\Bigg[\mathbf{\Phi}\Psi_1(T)\Bigg\{\int_0^{T}\Psi_1(s)^{-1}\mathbf B(s)u(s)ds\\
&\qquad\qquad\qquad+\sum_{i=1}^N\int_0^{T}\Psi_1(s)^{-1}\mathbf D_iu(s)dW_i(s)\Bigg\}\Psi_2(T)\Big|\mathcal F_{\cdot}\Bigg],\quad \widetilde{\mathcal L}_3u(\cdot):=(\mathcal L_3u(\cdot))(0),\\
&\mathcal L_4(\overline{\xi},\overline{\eta})(\cdot):=\mathbb E\left[\int_{\cdot}^T\mathbf{M}(s)\Psi_1(s)\left(\overline{\xi}+\sum_{i=1}^N\int_0^{s}\Psi_1(s)^{-1}\overline{\sigma}_i(s)dW_i(s)\right)\Psi_2(s)ds\Big|\mathcal F_{\cdot}\right]\\
&\qquad\qquad\qquad+\mathbb E\left[\left(\mathbf{\Phi}\Psi_1(T)\left(\overline{\xi}+\sum_{i=1}^N\int_0^{T}\Psi_1(s)^{-1}\overline{\sigma}_i(s)dW_i(s)\right)+\overline{\eta}\right)\Psi_2(T)\Big|\mathcal F_{\cdot}\right],\\
&\widetilde{\mathcal L}_4(\overline{\xi},\overline{\eta}):=\mathcal L_4(\overline{\xi},\overline{\eta})(0).
\end{aligned}\right.
\end{equation}
Correspondingly, $\mathcal L_1^\ast,\ \widetilde{\mathcal L}_3^\ast$ are defined as the adjoint operators of $\mathcal L_1,\ \widetilde{\mathcal L}_3$ \emph{w.r.t.} the inner product $\langle \cdot,\cdot\rangle$ (see e.g. \cite{YZ1999}), respectively. That is $\forall\ \zeta_1\in L_{\FF}^2(0,T;\mathbb R^n)$, $\forall\ \zeta_2\in L_{\FF}^2(0,T;\mathbb R^m)$, $\forall\ \zeta_3\in L^2_{\cF_T}(\Omega;\mathbb R^m)$,
\begin{equation}\nonumber \left\{
\begin{aligned}
&\mathbb E\int_0^T\langle(\mathcal L_1u(\cdot))(t),\zeta_1(t) \rangle dt=\mathbb E\int_0^T\langle u(t),(\mathcal L_1^\ast \zeta_1(\cdot))(t) \rangle dt,\\
&\mathbb E\langle \widetilde{\mathcal L}_3u(\cdot),\zeta_2(t) \rangle =\mathbb E\int_0^T\langle u(t),(\widetilde{\mathcal L}_3^\ast\zeta_2(\cdot))(t) \rangle dt,\\
&\mathbb E\langle \widetilde{\mathcal L}_3u(\cdot),\zeta_3 \rangle =\mathbb E\int_0^T\langle u(t),(\widetilde{\mathcal L}_3^\ast\zeta_3 )(t) \rangle dt.
\end{aligned}\right.
\end{equation}
Given any admissible $u(\cdot)$, we express $\mathbf X,\ \mathbf Y$ as
\begin{equation}\nonumber \left\{
\begin{aligned}
&\mathbf X(\cdot)=(\mathcal L_1u(\cdot))(\cdot)+\mathcal L_2(\overline{\xi})(\cdot),\qquad \mathbf X(T)=\widetilde{\mathcal L}_1u(\cdot)+\widetilde{\mathcal L}_2(\overline{\xi}),\\
&\mathbf Y(\cdot)=(\mathcal L_3u(\cdot))(\cdot)+\mathcal L_4(\overline{\xi},\overline{\eta})(\cdot),\quad \mathbf Y(0)=\widetilde{\mathcal L}_3u(\cdot)+\widetilde{\mathcal L}_4(\overline{\xi},\overline{\eta}).
\end{aligned}\right.
\end{equation}
Hence, we rewrite the social cost as
\begin{equation}\nonumber
\begin{aligned}
2\mathcal{J}^{(N)}_{soc}(u)=\ &\mathbb{E}\int_{0}^{T}\Big[ \langle \mathbf{Q}\mathbf{X},\mathbf{X}\rangle +\langle \mathbf{R}u,u\rangle \Big]dt+\langle \mathbf \Gamma\mathbf{Y}(0),\mathbf{Y}(0)\rangle \\
=\ &\mathbb{E}\int_{0}^{T}\Big[\langle (\mathcal L_1^\ast\mathbf{Q}\mathcal L_1u(\cdot))(t),u(t)  \rangle+2\langle\mathcal L_1^\ast\mathbf{Q} \mathcal L_2(\overline{\xi})(t),u(t) \rangle\\
&+\langle\mathbf{Q}\mathcal L_2(\overline{\xi})(t),\mathcal L_2(\overline{\xi})(t) \rangle+\langle \mathbf Ru(t),u(t) \rangle+\langle(\widetilde{\mathcal L}_3^\ast\mathbf \Gamma\widetilde{\mathcal L}_3u(\cdot))(t),u(t) \rangle\\
&+2\langle\widetilde{\mathcal L}_3^\ast\mathbf \Gamma\widetilde{\mathcal L}_4(\overline{\xi},\overline{\eta})(t),u(t) \rangle\Big]dt+\langle \mathbf\Gamma\widetilde{\mathcal L}_4(\overline{\xi},\overline{\eta}),\widetilde{\mathcal L}_4(\overline{\xi},\overline{\eta})\rangle\\
=\ &\mathbb{E}\int_{0}^{T}\Big[\big\langle (\mathcal L_1^\ast\mathbf{Q}\mathcal L_1u(\cdot))(t)+\mathbf Ru(t)+(\widetilde{\mathcal L}_3^\ast\mathbf \Gamma\widetilde{\mathcal L}_3u(\cdot))(t),u(t)\big\rangle+ 2\big\langle\mathcal L_1^\ast\mathbf{Q} \mathcal L_2(\overline{\xi})(t)\\
&+\widetilde{\mathcal L}_3^\ast\mathbf \Gamma\widetilde{\mathcal L}_4(\overline{\xi},\overline{\eta})(t),u(t)\big\rangle+ \langle\mathbf{Q}\mathcal L_2(\overline{\xi})(t),\mathcal L_2(\overline{\xi})(t) \rangle\Big]dt+\langle \mathbf\Gamma\widetilde{\mathcal L}_4(\overline{\xi},\overline{\eta}),\widetilde{\mathcal L}_4(\overline{\xi},\overline{\eta})\rangle\\
:=\ &\langle \mathcal M_2u(\cdot),u(\cdot)\rangle+2\langle \mathcal M_1,u(\cdot)\rangle+\mathcal M_0,
  \end{aligned}
\end{equation}
where $\mathcal M_2$ is an $L^2$ bounded self-adjoint positive definite linear operator; $\mathcal M_1$ is an $L^2$ bounded operator and $M_0\in\mathbb R$; $\langle\cdot,\cdot\rangle$ denotes the inner product of different space.

\subsection{Agent $\mathcal A_i,\ 1\leq i\leq N$ perturbation}\label{Agent's perturbation}

Let $\widetilde u=(\widetilde u_1,\cdots,\widetilde u_N)$ denote the set of decentralized strategies given by
$$\widetilde u_i(t)=-R_{\theta_i}^{-1}(t)\big(B^\top(t) p_i(t)+D^\top(t) \overline{p}_i(t)+K^\top(t) q_i(t)\big),\ 1\leq i\leq N,$$
where
\begin{equation}\label{eq39}\left\{\begin{aligned}
&dX_i(t)=\left[A_{\theta_i}X_i-BR_{\theta_i}^{-1}\big(B^\top p_i+D^\top \overline{p}_i+K^\top q_i\big)+F\sum_{l=1}^K\pi_l\mathbb E\alpha_l\right]dt\\
&\qquad\qquad+\Big[-DR_{\theta_i}^{-1}\big(B^\top p_i+D^\top \overline{p}_i+K^\top q_i\big)+\sigma_i\Big]dW_i(t),\\
&dY_i(t)=-\left[H_{\theta_i}Y_i-KR_{\theta_i}^{-1}\big(B^\top p_i+D^\top \overline{p}_i+K^\top q_i\big)+LX_i+M\sum_{l=1}^K\pi_l\mathbb E\alpha_l\right]dt\\
&\qquad\qquad+Z_idW_i(t),\\
&dp_i(t)=-\Bigg[A_{\theta_i}^\top p_i+L^\top q_i+QX_i-(QS+ S^\top Q -S^\top QS)\sum_{l=1}^K\pi_l\mathbb E\alpha_l\\
&\qquad\qquad+\sum_{l=1}^K\pi_lF^\top Y_2^l+\sum_{l=1}^K \pi_l(F^\top \widehat Y_l-M^\top \widehat X_l)\Bigg]dt+\overline{p}_idW_i(t),\\
&dq_i(t)=H_{\theta_i}^\top q_idt,\\
&X_i(0)=\xi_i,\quad Y_i(T)=\Phi X_i(T)+\eta_i,\quad p_i(T)=\Phi^\top q_i(T),\quad q_i(0)=\Gamma Y_i(0)
\end{aligned}\right.\end{equation}
with $\alpha_l,\ Y_2^l,\ \widehat X_l,\ \widehat Y_l$ being given by (\ref{eq0915a}). Actually the wellposedness of \eqref{eq39} can be obtained similar to \eqref{CC} (or \eqref{eq23}) and we omit it. So do the following coupled FBSDEs (\ref{eq36}).

Correspondingly, the real state $(\widetilde X_1,\cdots,\widetilde X_N,\widetilde Y_1,\cdots,\widetilde Y_N)$ under the decentralized strategy satisfies
\begin{equation}\label{eq36}\left\{\begin{aligned}
&d\widetilde{X}_i(t)=\Big[A_{\theta_i}\widetilde{X}_i-BR_{\theta_i}^{-1}\big(B^\top p_i+D^\top \overline{p}_i+K^\top q_i\big)+F\widetilde{X}^{(N)}\Big]dt\\
&\qquad\qquad+\Big[-DR_{\theta_i}^{-1}\big(B^\top p_i+D^\top \overline{p}_i+K^\top q_i\big)+\sigma_i\Big]dW_i(t),\\
&d\widetilde{Y}_i(t)=-\Big[H_{\theta_i}\widetilde{Y}_i-KR_{\theta_i}^{-1}\big(B^\top p_i+D^\top \overline{p}_i+K^\top q_i\big)+L\widetilde{X}_i+M\widetilde{X}^{(N)}\Big]dt+\widetilde{Z}_{i\cdot}dW(t),\\
&\widetilde{X}_i(0)=\xi_i,\quad \widetilde{Y}_i(T)=\Phi \widetilde{X}_i(T)+\eta_i,
\end{aligned}\right.\end{equation}
and $\widetilde X^{(N)}(\cdot)=\frac{1}{N}\sum_{i=1}^N\widetilde X_i(\cdot)$.
For $1\leq j\leq N$, define the perturbation as
$$\delta u_j=u_j-\widetilde u_j,\quad \delta X_j=\breve X_j-\widetilde X_j,\quad \delta Y_j=\breve Y_j-\widetilde Y_j,\quad \Delta\mathcal J_j=\mathcal J_j(u_j,\widetilde u_{-j})-\mathcal J_j(\widetilde u_j,\widetilde u_{-j}).$$
It should be noticed that hereafter the notation $(\breve{X}_j,\breve{Y}_j),\ 1\leq j\leq N$ stands for the state of $\mathcal A_j$ when applying an alternative strategy $u_j$ while $\mathcal A_l,l\neq j$ applies $\widetilde u_l$.
Similar to the computations in Section \ref{p-b-p optimality}, we have
\begin{equation}\label{eq42}\begin{aligned}
&\Delta \mathcal J_{soc}^{(N)}=\mathbb E\Bigg\{\int_0^T\Bigg[\left\langle Q\widetilde X_i,\delta X_i\right\rangle-\left\langle \left(QS+ S^\top Q -S^\top QS\right)\sum_{l=1}^K\pi_l\mathbb E\alpha_l,\delta X_i\right\rangle\\
  &\qquad\qquad+\sum_{l=1}^K\left\langle \pi_lF^\top Y_2^l,\delta X_i\right\rangle+\sum_{l=1}^K\left\langle \pi_l\left(F^\top \mathbb E\mathbf{Y}_l-M^\top \mathbf{X}_l\right),\delta X_i\right\rangle+\left\langle R_{\theta_i}\widetilde u_i,\delta u_i\right\rangle\Bigg]dt\\
  &\qquad\qquad+\left\langle \Gamma\widetilde Y_i(0),\delta Y_i(0)\right\rangle\Bigg\} +\sum_{l=1}^9\varepsilon_l,
\end{aligned}\end{equation}
where
\begin{equation}\nonumber\left\{\begin{aligned}
&\varepsilon_1=\mathbb E\int_0^T\left\langle \left(QS+ S^\top Q- S^\top Q S\right)\left(\sum_{l=1}^K\pi_l\mathbb E\alpha_l-\widetilde X^{(N)}\right),N\delta X^{(N)}\right\rangle dt,\\
&\varepsilon_2=\sum_{k=1}^K\mathbb E\int_0^T\left\langle\left (QS+ S^\top Q-S^\top QS\right)\sum_{l=1}^K\pi_l\mathbb E\alpha_l,X_k^{**}-\delta X_{(k)}\right\rangle dt,\\
&\varepsilon_3=\sum_{k=1}^K\mathbb E\int_0^T\frac{1}{N_k}\sum_{j\in\mathcal I_k,j\neq i}\left\langle Q\widetilde X_j,N_k\delta X_j-X^*_j\right\rangle dt,\\
&\varepsilon_4=\sum_{k=1}^K\frac{1}{N_k}\sum_{j\in\mathcal I_k,j\neq i}\left\langle \Gamma\widetilde Y_j(0),N_k\delta Y_j(0)-Y^*_j(0)\right\rangle,\\
&\varepsilon_{5}=\sum_{k=1}^K\mathbb E\int_0^T\left\langle \pi_kM^\top \mathbf X_k-\frac{1}{N_k}\sum_{j\in\mathcal I_k,j\neq i}\pi_kM^\top X_1^j,\delta X_i\right\rangle dt,\\
&\varepsilon_{6}=\sum_{k=1}^K\mathbb E\int_0^T\left\langle -\pi_kF^\top \mathbb E\mathbf Y_k+\frac{1}{N_k}\sum_{j\in\mathcal I_k,j\neq i}\pi_kF^\top Y_1^j,\delta X_i\right\rangle dt,\\
%\end{aligned}\right.\end{equation}
%\begin{equation}\nonumber\left\{\begin{aligned}
&\varepsilon_{7}=\sum_{k=1}^K\mathbb E\int_0^T\left\langle \sum_{l=1}^K\pi_lM^\top\mathbf{X}_l-\sum_{l=1}^K\frac{\pi_l}{N_l}\sum_{j\in\mathcal I_l,j\neq i}M^\top X_1^j,X_k^{**}\right\rangle dt,\\
&\varepsilon_{8}=\sum_{k=1}^K\mathbb E\int_0^T\left\langle -\sum_{l=1}^K\pi_lF^\top\mathbb E\mathbf{Y}_l+\sum_{l=1}^K\frac{\pi_l}{N_l}\sum_{j\in\mathcal I_l,j\neq i}F^\top Y_1^j,X_k^{**}\right\rangle dt,\\
&\varepsilon_{9}=-\sum_{k=1}^K\mathbb E\int_0^T\frac{I_{\mathcal I_k}(i)}{N_k}\left\langle \pi_kF^\top Y_2^k,\delta X_i\right\rangle dt.
\end{aligned}\right.\end{equation}

Now, we consider the situation that $\mathcal A_i,\ 1\leq i\leq N$ applies an alternative strategy $u_i$ while $\mathcal A_j,j\neq i$ applies $\widetilde u_j$. The real state with the $i^{th}$ agent's perturbation is
\begin{equation}\left\{\label{eq40}\begin{aligned}
d\breve X_i(t)=\ &\Big[A_{\theta_i}\breve X_i+Bu_i+F\breve X^{(N)}\Big]dt+\Big[Du_i+\sigma_i\Big]dW_i(t),\\
d\breve Y_i(t)=\ &-\Big[H_{\theta_i}\breve Y_i+Ku_i+L\breve X_i+M\breve X^{(N)}\Big]dt+\breve Z_{i\cdot}dW(t),\\
\breve X_i(0)=\ &\xi_i,\quad \breve Y_i(T)=\ \Phi \breve X_i(T)+\eta_i,
\end{aligned}\right.\end{equation}
and
\begin{equation}\left\{\label{eq41}\begin{aligned}
d\breve X_j(t)=\ &\Big[A_{\theta_j}\breve X_j-BR_{\theta_j}^{-1}\big(B^\top p_j+D^\top \overline{p}_j+K^\top q_j\big)+F\breve X^{(N)}\Big]dt\\
&\qquad+\Big[-DR_{\theta_j}^{-1}\big(B^\top p_j+D^\top \overline{p}_j+K^\top q_j\big)+\sigma_j\Big]dW_i(t),\\
d\breve Y_j(t)=\ &-\Big[H_{\theta_j}\breve Y_j-KR_{\theta_j}^{-1}\big(B^\top p_j+D^\top \overline{p}_j+K^\top q_j\big)+L\breve X_j+M\breve X^{(N)}\Big]dt\\
&\qquad+\breve Z_{j\cdot}dW(t),\\
\breve X_j(0)=\ &\xi_j,\quad \breve Y_j(T)=\ \Phi \breve X_j(T)+\eta_j,\ 1\leq j\leq N,\ j\neq i,
\end{aligned}\right.\end{equation}
where $\breve X^{(N)}=\frac{1}{N}\sum_{i=1}^N\breve X_i$.

To obtain the asymptotic optimality, we  first derive some estimations. In all the proofs hereafter, $C$ will denote a nonnegative constant and its value may change from line to line. Similar to the proof of \cite[Lemma 5.1]{HHN2018}, by virtue of estimations of FBSDE, we derive
\begin{lemma}\label{5.1}
Under \emph{(}A1\emph{)}-\emph{(}A4\emph{)}, there exists a constant $C\geq 0$ independent of $N$ such that
\begin{equation}\label{estimation-1}\begin{aligned}
&\sum_{l=1}^K\mathbb E\sup_{0\leq t\leq T}\Big[|\alpha_l(t)|^2+|\widetilde{\alpha}_l(t)|^2+|\beta_l(t)|^2+|\widetilde{\beta}_l(t)|^2+|\breve X_l(t)|^2+|\breve Y_l(t)|^2+|\vartheta_l(t)|^2\Big]\\
&+\sup_{1\leq i\leq N}\mathbb E\sup_{0\leq t\leq T}\Big[|\widetilde X_i(t)|^2+|\widetilde Y_i(t)|^2\Big]+\sum_{l=1}^K\mathbb E\int_0^T\Big[|\gamma_l(t)|^2+|\widetilde{\gamma}_l(t)|^2+|\breve Z_l(t)|^2\Big]dt\leq C.
\end{aligned}\end{equation}
\end{lemma}

Similar to Lemma \ref{5.1}, by the $L^2$ boundness of $u_i,\xi_i,\eta_i,\sigma_i$ and $\xi_j,\eta_j,\sigma_j,p_j,\overline{p}_j,q_j$ $(1\leq j\leq N,j\neq i)$, we have
 \begin{equation}\label{estimation-2}
\sup_{1\leq i\leq N}\mathbb E\sup_{0\leq t\leq T}\left[\left|\breve X_i(t)\right|^2+\left|\breve Y_i(t)\right|^2\right]\leq C.
\end{equation}

\begin{lemma}\label{lemma5.2}
Under \emph{(}A1\emph{)}-\emph{(}A4\emph{)}, there exists a constant $C\geq 0$ independent of $N$ such that
\begin{equation}\label{estimation-3}
\mathbb E\sup_{0\leq t\leq T}\left|\widetilde X^{(N)}(t)-\sum_{l=1}^K\pi_l\mathbb E\alpha_l(t)\right|^2\leq \frac{C}{N}+C\epsilon_N^2,
\end{equation}
where $\epsilon_N=\sup_{1\leq l\leq K}\left|\pi_l^{(N)}-\pi_l\right|.$
\end{lemma}
\begin{proof}
For $1\leq k\leq K$, state average of the $k$-type agent is defined by
$$
\widetilde X^{(k)}:=\frac{1}{N_k}\sum_{j\in\mathcal I_k}\widetilde X_j,\quad \widetilde Y^{(k)}:=\frac{1}{N_k}\sum_{j\in\mathcal I_k}\widetilde Y_j.
$$
thus
\begin{equation}\nonumber\left\{\begin{aligned}
&d\widetilde{X}^{(k)}(t)=\Bigg[A_k\widetilde{X}^{(k)}-\frac{1}{N_k}\sum_{j\in\mathcal I_k}BR_k^{-1}\big(B^\top p_j+D^\top \overline{p}_j+K^\top q_j\big)+F\widetilde{X}^{(N)}\Bigg]dt\\
&\qquad\qquad\qquad+\frac{1}{N_k}\sum_{j\in\mathcal I_k}\Big[-DR_k^{-1}\big(B^\top p_j+D^\top \overline{p}_j+K^\top q_j\big)+\sigma_j\Big]dW_j(t),\\
&d\widetilde{Y}^{(k)}(t)=-\Bigg[H_k\widetilde{Y}^{(k)}-\frac{1}{N_k}\sum_{j\in\mathcal I_k}KR_k^{-1}\big(B^\top p_j+D^\top \overline{p}_j+K^\top q_j\big)+L\widetilde{X}^{(k)}\\
&\qquad\qquad\qquad+M\widetilde{X}^{(N)}\Bigg]dt+\frac{1}{N_k}\sum_{j\in\mathcal I_k}\widetilde{Z}_{j\cdot}dW(t),\\
&\widetilde{X}^{(k)}(0)=\frac{1}{N_k}\sum_{j\in\mathcal I_k}\xi_j,\quad \widetilde{Y}^{(k)}(T)=\Phi \widetilde{X}^{(k)}(T)+\frac{1}{N_k}\sum_{j\in\mathcal I_k}\eta_j.
\end{aligned}\right.\end{equation}
Noticing that
\begin{equation}\nonumber\left\{\begin{aligned}
d\mathbb E\alpha_k(t)=\ &\Bigg[A_k\mathbb E\alpha_k-BR_k^{-1}\mathbb E\big(B^\top \mathbf P_k+D^\top \overline{\mathbf p}_k+K^\top \mathbf Q_k\big)+F\sum_{l=1}^K\pi_l\mathbb E\alpha_l \Bigg]dt, \\
\mathbb E\alpha_k(0)=\ &\mathbb E\xi^{(k)},
\end{aligned}\right.\end{equation}
we have
\begin{equation*}\left\{\begin{aligned}
&d\Big(\widetilde X^{(k)}(t)-\mathbb E\alpha_k(t)\Big)=\Bigg[A_k\Big(\widetilde X^{(k)}-\mathbb E\alpha_k\Big)-\frac{1}{N_k}\sum_{j\in\mathcal I_k}BR_k^{-1}\Big( B^\top p_j+D^\top \overline{p}_j\\
&\qquad\qquad+K^\top q_j-\mathbb E\big(B^\top \mathbf P_k+D^\top \overline{\mathbf p}_k+K^\top \mathbf Q_k\big)\Big)+F\left(\widetilde X^{(N)}-\sum_{l=1}^K\pi_l\mathbb E\alpha_l\right)\Bigg]dt\\
&\qquad\qquad+\frac{1}{N_k}\sum_{j\in\mathcal I_k}\Big[-DR_k^{-1}\big(B^\top p_j+D^\top \overline{p}_j+K^\top q_j\big)+\sigma_j\Big]dW_j(t),\\
&\Big(\widetilde X^{(k)}-\mathbb E\alpha_k\Big)(0)=\frac{1}{N_k}\sum_{j\in\mathcal I_k}\xi_j-\mathbb E\xi^{(k)}.
\end{aligned}\right.\end{equation*}
Here, we apply $\mathbf P_k,\overline{\mathbf  P}_k,\mathbf Q_k$ to denote the $k$-type representative when the expectations are involved as $(\mathbf X_k,\mathbf Y_k)$ before. Under (A2), for $1\leq k\leq K$, $\{\xi_j,j\in\mathcal I_k\}$ are identically independent distributed (i.i.d). Notice that $(p_j(\cdot),\overline{p}_j(\cdot))\in \mathcal F_t^j,\ q_j(\cdot)\in \mathcal F_0^j$. Therefore $\{(p_j,\overline{p}_j),j\in\mathcal I_k\}$ are i.i.d and $\{q_j,j\in\mathcal I_k\}$ are deterministic. Using Cauchy-Schwartz inequality, Burkholder-Davis-Gundy inequality and estimations of SDE, we derive
\small\begin{equation*}\begin{aligned}
&\mathbb E\sup_{0\leq s\leq t}\Big|\widetilde X^{(k)}-\mathbb E\alpha_k\Big|^2\\
\leq\ &\mathbb E\left|\frac{1}{N_k}\sum_{j\in\mathcal I_k}\left(\xi_j-\mathbb E\xi^{(k)}\right)\right|^2+C\mathbb E\int_0^t\left[\left|\widetilde X^{(k)}-\mathbb E\alpha_k\right|^2+\left|\widetilde X^{(N)}-\sum_{l=1}^K\pi_l\mathbb E\alpha_l\right|^2\right]ds\\
&+C\mathbb E\int_0^t\left|\frac{1}{N_k}\sum_{j\in\mathcal I_k}\Big( B^\top p_j+D^\top \overline{p}_j+K^\top q_j-\mathbb E\big(B^\top \mathbf P_k+D^\top \overline{\mathbf p}_k+K^\top \mathbf Q_k\big)\Big)\right|^2ds\\
&+C\mathbb E\int_0^t\frac{1}{N_k}\sum_{j\in\mathcal I_k}\left|-DR_k^{-1}\big(B^\top p_j+D^\top \overline{p}_j+K^\top q_j\big)+\sigma_j\right|^2ds\\
\leq\ &C\mathbb E\int_0^t\left|\widetilde X^{(k)}-\mathbb E\alpha_k\right|^2ds+C\mathbb E\int_0^t\left|\widetilde X^{(N)}-\sum_{l=1}^K\pi_l\mathbb E\alpha_l\right|^2ds+\frac{C}{N_k} .
\end{aligned}\end{equation*}\normalsize
Gronwall inequality implies that
\begin{equation*}\begin{aligned}
\mathbb E\sup_{0\leq s\leq t}\Big|\widetilde X^{(k)}-\mathbb E\alpha_k\Big|^2
\leq C\mathbb E\int_0^t\left|\widetilde X^{(N)}-\sum_{l=1}^K\pi_l\mathbb E\alpha_l\right|^2ds+\frac{C}{N_k}.
\end{aligned}\end{equation*}
Since
\begin{equation*}\begin{aligned}
\widetilde X^{(N)}-\sum_{l=1}^K\pi_l\mathbb E\alpha_l=&\sum_{l=1}^K\Big(\pi_l^{(N)}\widetilde X^{(l)}-\pi_l\mathbb E\alpha_l\Big)\\
=&
\sum_{l=1}^K\pi_l^{(N)}\Big(\widetilde X^{(l)}-\mathbb E\alpha_l\Big)+\sum_{l=1}^K\Big(\pi_l^{(N)}-\pi_l\Big)\mathbb E\alpha_l,
\end{aligned}\end{equation*}
we get
\begin{equation*}\begin{aligned}
&\mathbb E\sup_{0\leq s\leq t}\left|\widetilde X^{(N)}-\sum_{l=1}^K\pi_l\mathbb E\alpha_l\right|^2\\
\leq&\
C\sum_{l=1}^K\mathbb E\sup_{0\leq s\leq t}\Big|\widetilde X^{(l)}-\mathbb E\alpha_l\Big|^2+C\epsilon_N^2\leq C\mathbb E\int_0^t\left|\widetilde X^{(N)}-\sum_{l=1}^K\pi_l\mathbb E\alpha_l\right|^2ds+\frac{C}{N}+C\epsilon_N^2.
\end{aligned}\end{equation*}
Therefore, the result follows from Gronwall inequality.
\end{proof}

By Lemma \ref{lemma5.2}, we easily derive the following result.
\begin{lemma}\label{lemma5.22}
Under \emph{(}A1\emph{)}-\emph{(}A4\emph{)}, there exists a constant $C\geq 0$ independent of $N$ such that
\begin{equation}\label{estimation-33}
\sup_{1\leq j\leq N}\left[\mathbb E\sup_{0\leq t\leq T}\big|X_j(t)-\widetilde X_j(t)\big|^2+\mathbb E\sup_{0\leq t\leq T}\big|Y_j(t)-\widetilde Y_j(t)\big|^2\right]\leq \frac{C}{N}+C\epsilon_N^2.
\end{equation}
\end{lemma}
\begin{lemma}\label{Lemma5}
Under \emph{(}A1\emph{)}-\emph{(}A4\emph{)}, there exists a constant $C\geq 0$ independent of $N$ such that
\begin{equation}\begin{aligned}\label{estimation-5}
\sup_{1\leq j\leq N,j\neq i}\Bigg[\mathbb E\sup_{0\leq t\leq T}|\delta X_j(t)|^2+\mathbb E\sup_{0\leq t\leq T}|\delta Y_j(t)|^2&+\mathbb E\int_0^T\sum_{l=1}^N\big|\delta Z_{jl}(t)\big|^2dt\Bigg]\\
&\leq\frac{C}{N^2}\Bigg(1+\mathbb E\int_0^T|\delta u_i|^2ds\Bigg).
\end{aligned}\end{equation}
\end{lemma}
\begin{proof}
According to \eqref{eq6}-\eqref{eq9}, it yields
\begin{equation*}\begin{aligned}
&\mathbb E\sup_{0\leq s\leq t}|\delta X_i|^2\leq C\Bigg(1+\mathbb E\int_0^T|\delta u_i|^2ds\Bigg)+C\mathbb E\int_0^t|\delta X_i|^2ds+C\mathbb E\int_0^t|\delta X^{(N)}|^2ds,\\
&\mathbb E\sup_{0\leq s\leq t}|\delta Y_i|^2+\mathbb E\int_0^t\sum_{l=1}^N\big|\delta Z_{il}\big|^2ds\leq C\mathbb E|\delta X_i(T)|^2+C\Bigg(1+\mathbb E\int_0^T|\delta u_i|^2ds\Bigg)\\
&\qquad\qquad\qquad\qquad+C\mathbb E\int_0^t|\delta Y_i|^2ds+C\mathbb E\int_0^t|\delta X_i|^2ds+C\mathbb E\int_0^t\left|\delta X^{(N)}\right|^2ds,
\end{aligned}\end{equation*}
for $j\neq i$,
\begin{equation*}\begin{aligned}
&\mathbb E\sup_{0\leq s\leq t}|\delta X_j|^2\leq C\mathbb E\int_0^t|\delta X_j|^2ds+C\mathbb E\int_0^t\left|\delta X^{(N)}\right|^2ds,\\
&\mathbb E\sup_{0\leq s\leq t}|\delta Y_j|^2+\mathbb E\int_0^t\sum_{l=1}^N\big|\delta Z_{jl}\big|^2ds\\
&\qquad\leq C\mathbb E|\delta X_j(T)|^2+C\mathbb E\int_0^t|\delta Y_j|^2ds+C\mathbb E\int_0^t|\delta X_j|^2ds+C\mathbb E\int_0^t\left|\delta X^{(N)}\right|^2ds,
\end{aligned}\end{equation*}
and
\begin{equation*}\begin{aligned}
&\mathbb E\sup_{0\leq s\leq t}|\delta X_{(k)}|^2\leq C\mathbb E\int_0^t|\delta X_{(k)}|^2ds+CN^2\mathbb E\int_0^t|\delta X^{(N)}|^2ds.
\end{aligned}\end{equation*}
Noticing that $$\delta X^{(N)}=\frac{1}{N}\delta X_i+\frac{1}{N}\sum_{l=1}^K\delta X_{(l)},$$
we arrive at
\begin{equation*}\begin{aligned}
&\mathbb E\sup_{0\leq s\leq t}|\delta X_i|^2\leq C\Bigg(1+\mathbb E\int_0^T|\delta u_i|^2ds\Bigg)+C\mathbb E\int_0^t|\delta X_i|^2ds+\frac{C}{N^2}\sum_{l=1}^K\mathbb E\int_0^t\left|\delta X_{(l)}\right|^2ds,
\end{aligned}\end{equation*}
and
\begin{equation*}\begin{aligned}
&\mathbb E\sup_{0\leq s\leq t}\left|\delta X_{(k)}\right|^2\leq C\mathbb E\int_0^t|\delta X_{(k)}|^2ds+C\mathbb E\int_0^t|\delta X_i|^2ds+
 C\sum_{l=1}^K\mathbb E\int_0^t\left|\delta X_{(l)}\right|^2ds.
\end{aligned}\end{equation*}
Therefore, it follows from Gronwall inequality that
\begin{equation*}\begin{aligned}
\mathbb E\sup_{0\leq s\leq t}|\delta X_i|^2+\sum_{l=1}^K\mathbb E\sup_{0\leq s\leq t}\left|\delta X_{(l)}\right|^2\leq C\Bigg(1+\mathbb E\int_0^T|\delta u_i|^2ds\Bigg).
\end{aligned}\end{equation*}
Thus,
\begin{equation*}
\mathbb E\sup_{0\leq s\leq t}\left|\delta X^{(N)}\right|^2\leq\frac{C}{N^2}\Bigg(1+\mathbb E\int_0^T|\delta u_i|^2ds\Bigg).
\end{equation*}
Using Gronwall inequality again,
we have
\begin{equation*}\begin{aligned}
&\sup_{1\leq j\leq N,j\neq i}\Bigg[\mathbb E\sup_{0\leq s\leq t}|\delta X_j|^2+\mathbb E\sup_{0\leq s\leq t}|\delta Y_j|^2+\mathbb E\int_0^t\sum_{l=1}^N\big|\delta Z_{jl}\big|^2ds\Bigg]\\
&\leq\frac{C}{N^2}\Bigg(1+\mathbb E\int_0^T|\delta u_i|^2ds\Bigg).
\end{aligned}\end{equation*}
\end{proof}

%\begin{remark}
%Note that in \eqref{estimation-5}, the upper bound depends on $\mathbb E\int_0^T|\delta u_i|^2ds$. However, when studying the asymptotic optimality, we only need to consider the perturbations satisfying \eqref{condition for perturbation} below. Hence, in Section \ref{asymptotic optimality} when applying Lemma \ref{Lemma5}, similar estimation still holds while the upper bound is $\frac{C}{N^2}$ and $C$ is a general constant.
%\end{remark}

\begin{lemma}\label{lemma5.4}
Under \emph{(}A1\emph{)}-\emph{(}A4\emph{)}, there exists a constant $C\geq 0$ independent of $N$ such that
\begin{equation}\label{estimation-6}
\sum_{l=1}^K\mathbb E\sup_{0\leq t\leq T}\left|X_l^{**}(t)-\delta X_{(l)}(t)\right|^2\leq \Big(\frac{C}{N^2}+C\epsilon_N^2\Big)  \Bigg(1+\mathbb E\int_0^T|\delta u_i|^2ds\Bigg),
\end{equation}
for $j\in\mathcal I_k,\ j\neq i$, $1\leq k\leq K$,
\begin{equation}\label{estimation-7}
\mathbb E\sup_{0\leq t\leq T}\left|N_k\delta X_j(t)-X^*_j(t)\right|^2\leq \Big(\frac{C}{N^2}+C\epsilon_N^2\Big)  \Bigg(1+\mathbb E\int_0^T|\delta u_i|^2ds\Bigg),
\end{equation}
\begin{equation}\label{estimation-77}
\mathbb E\sup_{0\leq t\leq T}\left|N_k\delta Y_j(t)-Y^*_j(t)\right|^2\leq \Big(\frac{C}{N^2}+C\epsilon_N^2\Big)  \Bigg(1+\mathbb E\int_0^T|\delta u_i|^2ds\Bigg).
\end{equation}
\end{lemma}
\begin{proof}
First,
\begin{equation*}\left\{\begin{aligned}
&d(X_k^{**}-\delta X_{(k)})=\Bigg[A_k(X_k^{**}-\delta X_{(k)})+F\Bigg(\pi_k-\frac{N_k-I_{\mathcal I_k}(i)}{N}\Bigg)\delta X_i\\
&\qquad\qquad\qquad\qquad+F\pi_k\sum_{l=1}^K(X_l^{**}-\delta X_{(l)})+F\Bigg(\pi_k-\frac{N_k-I_{\mathcal I_k}(i)}{N}\Bigg)\sum_{l=1}^K\delta X_{(l)}\Bigg]dt,\\
&(X_k^{**}-\delta X_{(k)})(0)=0,
\end{aligned}\right.\end{equation*}
and for $j\in\mathcal I_k,\ j\neq i$,
\begin{equation*}\left\{\begin{aligned}
&d(X_j^*-N_k\delta X_j)=\Bigg[A_{k}(X_j^*-N_k\delta X_j)+F\Big(\pi_k-\pi_k^{(N)}\Big)\delta X_i+F\Big(\pi_k-\pi_k^{(N)}\Big)\sum_{l=1}^KX_l^{**}\\
&\qquad\qquad\qquad\qquad+F\pi_k^{(N)}\sum_{l=1}^K(X_l^{**}-\delta X_{(l)})\Bigg]dt,\\
&(X_j^*-N_k\delta X_j)(0)=0,
\end{aligned}\right.\end{equation*}and
\begin{equation*}\left\{\begin{aligned}
&d(Y_j^*-N_k\delta Y_j)=-\Bigg[H_{k}(Y_j^*-N_k\delta Y_j)+L(X_j^*-N_k\delta X_j)+M\Big(\pi_k-\pi_k^{(N)}\Big)\delta X_i\\
&\quad\ \ +M\Big(\pi_k-\pi_k^{(N)}\Big)\sum_{l=1}^KX_l^{**}+M\pi_k^{(N)}\sum_{l=1}^K(X_l^{**}-\delta X_{(l)})\Bigg]dt+(Z_{j\cdot}^*-N_k\delta Z_{j\cdot})dW(t),\\
&(Y_j^*-N_k\delta Y_j)(T)=\Phi(X_j^*(T)-N_k\delta X_j(T)).
\end{aligned}\right.\end{equation*}
According to Burkholder-Davis-Gundy inequality,
\small\begin{equation*}\begin{aligned}
\mathbb E\sup_{0\leq s\leq t}|X_k^{**}-\delta X_{(k)}|^2 \leq\ &C\mathbb E\int_0^t|X_k^{**}-\delta X_{(k)}|^2ds+C\sum_{l=1}^K\mathbb E\int_0^t|X_l^{**}-\delta X_{(l)}|^2ds\\
&\quad+\Big(\frac{C}{N^2}+C\epsilon_N^2\Big)\mathbb E\int_0^t\left[|\delta X_i|^2+\sum_{l=1}^K|\delta X_{(l)}|^2\right]ds\\
\leq \ & C\mathbb E\int_0^t|X_k^{**}-\delta X_{(k)}|^2ds+C\sum_{l=1}^K\mathbb E\int_0^t|X_l^{**}-\delta X_{(l)}|^2ds\\
&\quad+\Big(\frac{C}{N^2}+C\epsilon_N^2\Big)  \Bigg(1+\mathbb E\int_0^T|\delta u_i|^2ds\Bigg).
\end{aligned}\end{equation*}\normalsize
Thus,
\begin{equation*}\begin{aligned}
\sum_{l=1}^K\mathbb E\sup_{0\leq s\leq t}|X_l^{**}-\delta X_{(l)}|^2
\leq& C\sum_{l=1}^K\mathbb E\int_0^t|X_l^{**}-\delta X_{(l)}|^2\\
&\quad+\Big(\frac{C}{N^2}+C\epsilon_N^2\Big)  \Bigg(1+\mathbb E\int_0^T|\delta u_i|^2ds\Bigg).
\end{aligned}\end{equation*}
By virtue of Gronwall inequality, we have
$$\sum_{l=1}^K\mathbb E\sup_{0\leq t\leq T}|X_l^{**}(t)-\delta X_{(l)}(t)|^2\leq\Big(\frac{C}{N^2}+C\epsilon_N^2\Big)  \Bigg(1+\mathbb E\int_0^T|\delta u_i|^2ds\Bigg).$$
Similarly,
\begin{equation*}\begin{aligned}
\mathbb E\sup_{0\leq s\leq t}|X_j^*-N_k\delta X_j|^2 \leq& C\mathbb E\int_0^t|X_j^*-N_k\delta X_j|^2ds+C\sum_{l=1}^K\mathbb E\int_0^t|X_l^{**}-\delta X_{(l)}|^2ds\\
&\ +C\epsilon_N^2\mathbb E\int_0^t\left[|\delta X_i|^2+\sum_{l=1}^K|X_l^{**}|^2\right]ds.
%\leq \ & C\mathbb E\int_0^t|X_k^{**}-\delta X_{(k)}|^2ds+C\sum_{l=1}^K\mathbb E\int_0^t|X_l^{**}-\delta X_{(l)}|^2ds+\Big(\frac{C}{N^2}+C\epsilon_N^2\Big)  \Bigg(1+\mathbb E\int_0^T|\delta u_i|^2ds\Bigg).
\end{aligned}\end{equation*}
By the first equation of \eqref{eq11}, we derive
\begin{equation*}\begin{aligned}
&\mathbb E\sup_{0\leq s\leq t}|X_k^{**}|^2 \leq\ C\mathbb E\int_0^t|X_k^{**}|^2ds+C\sum_{l=1}^K\mathbb E\int_0^t|X_l^{**}|^2ds+C\mathbb E\int_0^t|\delta X_i|^2ds.
\end{aligned}\end{equation*}
It follows from Gronwall inequality that
\begin{equation*}\begin{aligned}
&\sum_{l=1}^K\mathbb E\sup_{0\leq s\leq t}|X_l^{**}|^2 \leq C\mathbb E\int_0^t|\delta X_i|^2ds.
\end{aligned}\end{equation*}
Then noticing \eqref{estimation-6},
\begin{equation*}\begin{aligned}
\mathbb E\sup_{0\leq s\leq t}|X_j^*-&N_k\delta X_j|^2\leq C\mathbb E\int_0^t|X_j^*-N_k\delta X_j|^2ds+C\sum_{l=1}^K\mathbb E\int_0^t|X_l^{**}-\delta X_{(l)}|^2ds\\
&\qquad\qquad+C\epsilon_N^2\mathbb E\int_0^t|\delta X_i|^2ds\\
&\leq C\mathbb E\int_0^t|X_j^*-N_k\delta X_j|^2ds+\Big(\frac{C}{N^2}+C\epsilon_N^2\Big)  \Bigg(1+\mathbb E\int_0^T|\delta u_i|^2ds\Bigg),
\end{aligned}\end{equation*}
which implies \eqref{estimation-7}. With the help of the estimations of BSDE, \eqref{estimation-77} is derived.
\end{proof}

Applying the above estimations, by the standard estimations of FBSDE, the $L^2$ boundness of $\xi_i,\eta_i,\sigma_i,p_i,\overline{p}_i,q_i$, \eqref{eq17} and \eqref{eq36}, we get the following result.
\begin{lemma}\label{Lemma7}
Under \emph{(}A1\emph{)}-\emph{(}A4\emph{)}, there exists a constant $C\geq0$ independent of $N$ such that
\begin{equation}\label{estimation-8}
\sum_{k=1}^K\mathbb E\sup_{0\leq t\leq T}\left|\mathbf{X}_k(t)-\frac{1}{N_k}\sum_{j\in\mathcal I_k,j\neq i}X_1^j(t)\right|^2\leq \frac{C}{N}+C\epsilon_N^2,
\end{equation}
\begin{equation}\label{estimation-9}
\sum_{k=1}^K\mathbb E\sup_{0\leq t\leq T}\left|\mathbb E{\mathbf{Y}}_k(t)-\frac{1}{N_k}\sum_{j\in\mathcal I_k,j\neq i}Y_1^j(t)\right|^2\leq \frac{C}{N}+C\epsilon_N^2.
\end{equation}
\end{lemma}
\begin{proof}
It follows from \eqref{eq17} that
\begin{equation}\nonumber\left\{\begin{aligned}
&d\Bigg(\frac{1}{N_k}\sum_{j\in\mathcal I_k,j\neq i}X_1^j\Bigg)=\frac{H_k^\top }{N_k}\sum_{j\in\mathcal I_k,j\neq i}X_1^j dt,\\
&d\Bigg(\frac{1}{N_k}\sum_{j\in\mathcal I_k,j\neq i}Y_1^j\Bigg)=\Bigg[- \frac{A_k^\top}{N_k}\sum_{j\in\mathcal I_k,j\neq i}Y_1^j+ \frac{L^\top}{N_k}\sum_{j\in\mathcal I_k,j\neq i}X_1^j-\frac{Q}{N_k}\sum_{j\in\mathcal I_k,j\neq i}\widetilde X_j\Bigg] dt\\
&\qquad\qquad\qquad\qquad\qquad+\frac{1}{N_k}\sum_{j\in\mathcal I_k,j\neq i}Z_1^{j\cdot}dW(t),\\
&\frac{1}{N_k}\sum_{j\in\mathcal I_k,j\neq i}X_1^j(0)=-\frac{\Gamma}{N_k}\sum_{j\in\mathcal I_k,j\neq i}\widetilde Y_j(0),\  \frac{1}{N_k}\sum_{j\in\mathcal I_k,j\neq i}Y_1^j(T)=- \frac{\Phi^\top}{N_k}\sum_{j\in\mathcal I_k,j\neq i}X_1^j(T).
\end{aligned}\right.\end{equation}
By the definition of $\mathbf{X}_k,\ \mathbf{Y}_k$, we have
\begin{equation}\nonumber\left\{\begin{aligned}
&d\mathbf{X}_k=H_k^\top \mathbf{X}_k dt,\\
&d\mathbb E\mathbf{Y}_k=\Big[-A_k^\top \mathbb E\mathbf{Y}_k+L^\top \mathbf{X}_k-Q\mathbb E\widetilde X_k\Big] dt,\\
&\mathbf{X}_k (0)=-\Gamma\mathbb E\widetilde Y_k(0),\quad \mathbb E\mathbf{Y}_k(T)=-\Phi^\top \mathbf{X}_k(T),\quad k=1,\cdots,K,
\end{aligned}\right.\end{equation}
where $\widetilde X_k,\ \widetilde Y_k$ denote the optimal states of $k-$type corresponding to \eqref{eq36} and $\mathbb E\widetilde{X}_k,$\\ $\mathbb E\widetilde{Y}_k$ satisfy
\begin{equation}\nonumber\left\{\begin{aligned}
&d\mathbb E\widetilde{X}_k=\Big[A_k\mathbb E\widetilde{X}_k-BR_k^{-1}\mathbb E\big(B^\top \mathbf{p}_k+D^\top \overline{\mathbf{p}}_k+K^\top \mathbf{q}_k\big)+F\mathbb E\widetilde{X}^{(N)}\Big]dt,\\
&d\mathbb E\widetilde{Y}_k=-\Big[H_k\mathbb E\widetilde{Y}_k-KR_k^{-1}\mathbb E\big(B^\top \mathbf{p}_k+D^\top \overline{\mathbf{p}}_k+K^\top \mathbf{q}_k\big)+L\mathbb E\widetilde{X}_k+M\mathbb E\widetilde{X}^{(N)}\Big]dt,\\
&\mathbb E\widetilde{X}_k(0)=\mathbb E\xi^{(k)},\quad \mathbb E\widetilde{Y}_k(T)=\Phi \mathbb E\widetilde{X}_k(T)+\mathbb E\eta_k.
\end{aligned}\right.\end{equation}
Recall the notations $\widetilde X^{(k)}$ and $\widetilde Y^{(k)}$ defined in the proof of Lemma \ref{lemma5.2}. Noticing $$\frac{1}{N_k}\sum_{j\in\mathcal I_k,j\neq i}\widetilde X_j=\widetilde X^{(k)}-\frac{I_{\mathcal I_k}(i)}{N_k}\widetilde X_i,\quad \frac{1}{N_k}\sum_{j\in\mathcal I_k,j\neq i}\widetilde Y_j=\widetilde Y^{(k)}-\frac{I_{\mathcal I_k}(i)}{N_k}\widetilde Y_i,$$
we have
\begin{equation*}\begin{aligned}
&\mathbb E\sup_{0\leq s\leq t}\left|\frac{1}{N_k}\sum_{j\in\mathcal I_k,j\neq i}X_1^j-\mathbf{X}_k\right|^2 \leq C\mathbb E\int_0^t\left|\frac{1}{N_k}\sum_{j\in\mathcal I_k,j\neq i}X_1^j-\mathbf{X}_k\right|^2ds\\
&\qquad\qquad\qquad+C\Bigg(\left|\widetilde Y^{(k)}(0)-\mathbb E\beta_k(0)\right|^2+|\mathbb E\widetilde Y_k(0)-\mathbb E\beta_k(0)|^2+\frac{1}{N_k^2}|\widetilde Y_i(0)|^2\Bigg),\\
&\mathbb E\sup_{t\leq s\leq T}\left|\frac{1}{N_k}\sum_{j\in\mathcal I_k,j\neq i}Y_1^j-\mathbb E\mathbf{Y}_k\right|^2\leq C\mathbb E\left|\frac{1}{N_k}\sum_{j\in\mathcal I_k,j\neq i}X_1^j(T)-\mathbf{X}_k(T)\right|^2\\
&\qquad\qquad+C\mathbb E\int_t^T\left|\frac{1}{N_k}\sum_{j\in\mathcal I_k,j\neq i}Y_1^j-\mathbb E\mathbf{Y}_k\right|^2ds+C\mathbb E\int_t^T\left|\frac{1}{N_k}\sum_{j\in\mathcal I_k,j\neq i}X_1^j-\mathbf{X}_k\right|^2ds\\
&\qquad\qquad+C\mathbb E\int_t^T\left(|\widetilde X^{(k)}-\mathbb E\alpha_k|^2+|\mathbb E\widetilde X_k-\mathbb E\alpha_k|^2+\frac{1}{N_k^2}|\widetilde X_i|^2\right)ds.
\end{aligned}\end{equation*}
By the $L^2$ boundness of $\xi_i,\eta_i,\sigma_i,p_i,\overline{p}_i,q_i$, it is easy to get $\mathbb E\sup_{0\leq s\leq t}\big(|\widetilde X_i|^2+|\widetilde Y_i|^2\big)\leq C$, $1\leq i\leq N$. In addition,
\begin{equation*}\begin{aligned}
&\sup_{0\leq s\leq t}|\mathbb E\widetilde X_k-\mathbb E\alpha_k|^2 \leq C\mathbb E\int_0^t\left|\widetilde X^{(N)}-\sum_{l=1}^K\pi_l\mathbb E\alpha_l\right|^2ds,\\
&\sup_{t\leq s\leq T}|\mathbb E\widetilde Y_k-\mathbb E\beta_k|^2 \leq C|\mathbb E\widetilde X_k(T)-\mathbb E\alpha_k(T)|^2+C\mathbb E\int_t^T|\mathbb E\widetilde X_k-\mathbb E\alpha_k|^2ds\\
&\qquad\qquad\qquad\qquad\qquad+C\mathbb E\int_t^T\left|\widetilde X^{(N)}-\sum_{l=1}^K\pi_l\mathbb E\alpha_l\right|^2ds.
\end{aligned}\end{equation*}
With the help of the proof of Lemma \ref{lemma5.2}, one gets
\begin{equation*}\begin{aligned}
&\mathbb E\sup_{0\leq t\leq T}\Big|\widetilde X^{(k)}-\mathbb E\alpha_k\Big|^2\leq \frac{C}{N}+C\epsilon_N^2,\quad \mathbb E\sup_{0\leq t\leq T}\Big|\widetilde Y^{(k)}-\mathbb E\beta_k\Big|^2\leq \frac{C}{N}+C\epsilon_N^2.
\end{aligned}\end{equation*}
Then \eqref{estimation-8} and \eqref{estimation-9} are obtained based on above inequalities, Lemma \ref{lemma5.2} and Gronwall inequality.
\end{proof}
\subsection{Asymptotic optimality}

To verify the asymptotic optimality, we just need to investigate the perturbation $u_{i}\in\mathcal U_i^c,\ 1\leq i\leq N$ satisfying $\mathcal J_{soc}^{(N)}(u_1,\cdots,u_N)\leq\mathcal J_{soc}^{(N)}(\widetilde u_1,$ $\cdots,\widetilde u_N).$ Obviously, $$\mathcal J_{soc}^{(N)}(\widetilde u_1,\cdots,\widetilde u_N)\leq CN,$$ where $C$ is a nonnegative constant independent of $N$. Therefore we need only to investigate the perturbation $u_i\in\mathcal U_i^{c}$ satisfying
\begin{equation}\label{condition for perturbation}
\sum_{i=1}^N\mathbb E\int_0^T|u_i|^2dt\leq CN.
\end{equation}
Let $\delta u_i=u_i-\widetilde u_i,\ 1\leq i\leq N$. Now we consider a perturbation  ${u} = \widetilde{u} + (\delta u_1,\cdots,\delta u_N):=\widetilde{ u}+ \delta u$. Recalling Lemma \ref{Lemma5} and Lemma \ref{lemma5.4}, there exists a constant $C\geq0$ independent of $N$ such that
\begin{equation*}\begin{aligned}
&\sup_{1\leq j\leq N,j\neq i}\Bigg[\mathbb E\sup_{0\leq t\leq T}|\delta X_j(t)|^2+\mathbb E\sup_{0\leq t\leq T}|\delta Y_j(t)|^2+\mathbb E\int_0^T\sum_{l=1}^N\big|\delta Z_{jl}(t)\big|^2dt\Bigg]
\leq\frac{C}{N^2},\\
&\sum_{l=1}^K\mathbb E\sup_{0\leq t\leq T}|X_l^{**}(t)-\delta X_{(l)}(t)|^2\leq \frac{C}{N^2}+C\epsilon_N^2,\\
&\mathbb E\sup_{0\leq t\leq T}|N_k\delta X_j(t)-X^*_j(t)|^2\leq \frac{C}{N^2}+C\epsilon_N^2  ,\\
&\mathbb E\sup_{0\leq t\leq T}|N_k\delta Y_j(t)-Y^*_j(t)|^2\leq \frac{C}{N^2}+C\epsilon_N^2  .
\end{aligned}\end{equation*}
Further, by Section \ref{Resc}, we have
\begin{equation}\nonumber
\begin{aligned}
2\mathcal{J}^{(N)}_{soc}(\widetilde{ u}+ \delta u)=\ &\langle \mathcal M_2(\widetilde{ u}+ \delta u),\widetilde{ u}+ \delta u\rangle+2\langle \mathcal M_1,\widetilde{ u}+ \delta u\rangle+\mathcal M_0\\
=\ &\langle \mathcal M_2\widetilde{ u},\widetilde{ u}\rangle+\langle \mathcal M_2\delta u,\delta u\rangle+2\langle \mathcal M_2\widetilde{ u},\delta u\rangle+2\langle \mathcal M_1,\widetilde{ u}\rangle+2\langle \mathcal M_1,\delta u\rangle+\mathcal M_0\\
=\ &2\mathcal{J}^{(N)}_{soc}(\widetilde{ u})+\langle \mathcal M_2\delta u,\delta u\rangle+2\langle \mathcal M_2\widetilde{ u},\delta u\rangle+2\langle \mathcal M_1,\delta u\rangle\\
=\ &2\mathcal{J}^{(N)}_{soc}(\widetilde{ u})+2\langle \mathcal M_2\widetilde{ u}+\mathcal M_1,\delta u\rangle+\langle \mathcal M_2\delta u,\delta u\rangle,
  \end{aligned}
\end{equation}
where $\mathcal M_2\widetilde{u} +\mathcal M_1 $ denotes the Fr\'{e}chet differential of ${\mathcal J}^{(N)}_{soc}$ corresponding to $\widetilde{u}$.

\begin{theorem}\label{asymptotic optimal}
Under \emph{(}A1\emph{)}-\emph{(}A4\emph{)}, $\widetilde u=(\widetilde u_1,\cdots,\widetilde u_N)$ is a $\Big(\frac{1}{\sqrt{N}}+\epsilon_N\Big)$-social decentralized optimal strategy, where $\epsilon_N=\sup_{1\leq l\leq K}\left|\pi_l^{(N)}-\pi_l\right|$.
\end{theorem}
\begin{proof}
By virtue of Cauchy-Schwarz inequality, we get
\begin{equation*}
  \begin{aligned}
&\mathcal{J}^{(N)}_{soc}(\widetilde{u} )-\mathcal{J}^{(N)}_{soc}(\widetilde{u} + \delta u)\\
\leq & \sqrt{\sum_{i=1}^N|\mathcal M_2\widetilde u+\mathcal M_1|^2\sum_{i=1}^N|\delta u_i|^2}-\frac{1}{2}\langle \mathcal M_2\delta u,\delta u\rangle
\leq|\mathcal M_2\widetilde u+\mathcal M_1|O(N).
\end{aligned}
\end{equation*}
Thus $$|\mathcal M_2\widetilde u+\mathcal M_1|=o(1)$$
ensures the asymptotic optimality.
According to Section \ref{Agent's perturbation}, we derive
\begin{equation}\begin{aligned}
&|\mathcal M_2\widetilde u+\mathcal M_1|\\
  =\ &\mathbb E\Bigg\{\int_0^T\Bigg[\langle Q\widetilde X_i,\delta X_i\rangle\!-\!\left\langle (QS+ S^\top Q -S^\top QS)\sum_{l=1}^K\pi_l\mathbb E\alpha_l,\delta X_i\right\rangle\!+\!\sum_{l=1}^K\langle \pi_lF^\top Y_2^l,\delta X_i\rangle
\\
&+\sum_{l=1}^K\langle \pi_l(F^\top \mathbb E\mathbf{Y}_l-M^\top \mathbb E\mathbf{X}_l),\delta X_i\rangle+\langle R_{\theta_i}\widetilde u_i,\delta u_i\rangle\Bigg]dt +\langle \Gamma\widetilde Y_i(0),\delta Y_i(0)\rangle\Bigg\}+\sum_{l=1}^9\varepsilon_l.
\end{aligned}\end{equation}
According to the optimality of $\widetilde u$,
\begin{equation*}\begin{aligned}
&\mathbb E\Bigg\{\int_0^T\Bigg[\langle Q\widetilde X_i,\delta X_i\rangle-\left\langle (QS+ S^\top Q -S^\top QS)\sum_{l=1}^K\pi_l\mathbb E\alpha_l,\delta X_i\right\rangle+\sum_{l=1}^K\langle \pi_lF^\top Y_2^l,\delta X_i\rangle
\\
&+\sum_{l=1}^K\langle \pi_l(F^\top \mathbb E\mathbf{Y}_l-M^\top \mathbb E\mathbf{X}_l),\delta X_i\rangle+\langle R_{\theta_i}\widetilde u_i,\delta u_i\rangle\Bigg]dt +\langle \Gamma\widetilde Y_i(0),\delta Y_i(0)\rangle\Bigg\}=0.
\end{aligned}\end{equation*}
Moreover, by Lemmas \ref{lemma5.2}-\ref{Lemma7}, we have $$\sum_{l=1}^{9}\varepsilon_l=
O\Big(\frac{1}{\sqrt{N}}+\epsilon_N\Big).$$
Therefore, $$|\mathcal M_2\widetilde u+\mathcal M_1|=O\Big(\frac{1}{\sqrt{N}}+\epsilon_N\Big).$$
\end{proof}

%The process of verifying asymptotic optimality can be illustrated by Figure \ref{f2} below.
%\begin{figure}[!h]
%\centering
%\setlength{\abovecaptionskip}{-0.1cm}
%\hspace*{0.4cm}\includegraphics[width=5.5in,height=2.7in]{pic2.eps}
%\caption{Process of verifying asymptotic optimality.}\label{f2}
%\end{figure}

\section{Conclusion}
\setcounter{equation}{0}
\renewcommand{\theequation}{\thesection.\arabic{equation}}

This paper focuses on solving an LQ stochastic optimization problem in MF social optima scheme while the dynamic is driven by FBSDE. An auxiliary LQ control problem is formulated and a decentralized strategy is obtained with the help of consistency condition system. We also develop a Riccati equation and a BSDE to decouple the MF-type FBSDE. At last, we verify the asymptotic optimality. In the future, enlightened by the first motivation in Section \ref{sec1.2} one possible research direction is to study the case that the dynamic satisfies a nonlinear system, which may be more valuable but more complicated than the LQ structure shown in this work. Another research problem is LQ MF social optima with partial observation, which may involve more applications in practice and bring more challenges in theory.

\end{document}